\documentclass[english]{amsart}
\usepackage[latin9]{inputenc}
\usepackage{amstext}
\usepackage{amsthm}
\usepackage{amssymb}

\makeatletter

\newcommand{\lyxmathsym}[1]{\ifmmode\begingroup\def\b@ld{bold}
	\text{\ifx\math@version\b@ld\bfseries\fi#1}\endgroup\else#1\fi}

\numberwithin{equation}{section}
\numberwithin{figure}{section}
\theoremstyle{plain}
\newtheorem{thm}{\protect\theoremname}[section]
\theoremstyle{definition}
\newtheorem{defn}[thm]{\protect\definitionname}
\theoremstyle{plain}
\newtheorem{lem}[thm]{\protect\lemmaname}
\theoremstyle{remark}
\newtheorem{rem}[thm]{\protect\remarkname}
\theoremstyle{remark}
\newtheorem*{rem*}{\protect\remarkname}
\theoremstyle{plain}

\makeatother

\usepackage{babel}
\providecommand{\corollaryname}{Corollary}
\providecommand{\definitionname}{Definition}
\providecommand{\lemmaname}{Lemma}
\providecommand{\remarkname}{Remark}
\providecommand{\theoremname}{Theorem}

\begin{document}
	
	\title[Multisolitons]{Multisolitons for the Defocusing Energy Critical Wave Equation with
		Potentials}
	
	\author{Gong Chen}
	
	\thanks{This work is part of the author\textquoteright s Ph.D. thesis at
		the University of Chicago.}
	
	\date{\today}
	
	\urladdr{http://www.math.uchicago.edu/\textasciitilde{}gc/}
	
	\email{gc@math.uchicago.edu}
	
	\address{Department of Mathematics, The University of Chicago, 5734 South
		University Avenue, Chicago, IL 60637, U.S.A.}
\begin{abstract}
We construct multisoliton solutions to the defocusing energy critical
wave equation with potentials in $\mathbb{R}^{3}$ based on regular
and reversed Strichartz estimates developed in \cite{GC3} for wave
equations with charge transfer Hamiltonians. We also show the
asymptotic stability of multisoliton solutions.  The multisoliton structures with both stable and unstable solitons are covered. Since each soliton
decays slowly with rate $\frac{1}{\left\langle x\right\rangle }$, the interactions among the solitons are strong.  Some reversed Strichartz estimates and local decay estimates for the charge transfer model are established to handle strong interactions. 

\end{abstract}

\maketitle
\tableofcontents{}
\section{Introduction}

In this paper, we consider multisoliton structures to the defocusing
energy critical wave equation with potentials in $\mathbb{R}^{3}$:
\begin{equation}
\partial_{tt}u-\Delta u+\sum_{j=1}^{m}V_{j}\left(x-\vec{v}_{j}t\right)u+u^{5}=0,\label{eq:maineq}
\end{equation}
where $V_{j}(x)$'s are rapidly decaying smooth potentials and $\left\{ \vec{v}_{j}\right\} $
is a set of distinct constant velocities such that 
\begin{equation}
\left|\vec{v}_{i}\right|<1,\,1\leq i\leq m.
\end{equation}
Based on both regular and reversed Strichartz estimates developed
in Chen \cite{GC3} for wave equations with charge transfer Hamiltonians,
we construct  purely multi-soliton solutions and establish the asymptotic
stability of the multisoliton solutions. 

\smallskip

To the author's knowledge, this model is the first one to produce multisoliton structures for wave equations
in $\mathbb{R}^{3}$. Unlike Klein-Gordon equations and wave equations
in higher dimensions, see C\^ote-Mu\~noz \cite{CM}, C\^ote-Martel \cite{CM1},
Jendrej \cite{JJ1,JJ2}, Martel-Merle \cite{MM}, in our case, the
static solutions to the associated elliptic equations decay slowly
like $\left\langle x\right\rangle ^{-1}$. It is of crucial importance
to understand the multisoliton structure in order to establish the
soliton resolution. In fact, if we remove the potentials and replace
the positive sign in front of the nonlinearity by the negative sign,
the equation becomes the well-known focusing energy critical wave
equation. Duyckaerts, Jia, Kenig and Merle establish the soliton resolution
(along a well-chosen time sequence) in \cite{DKM,DJKM}. But to construct
the multisoliton in this case is open. For higher dimensions cases,
Martel and Merle construct the multisoliton in dimension higher than
$5$ by the energy method in \cite{MM}. They point out that the
slow decay of the ground state is the obstruction to obtain a multisoliton
in $\mathbb{R}^{3}$. Although the structure of our model is different
from the pure-power nonlinear equation, the construction in this paper
illustrates that we can overcome the slow decay. But the zero eigenfunctions
and resonances for the linearized operator from the pure-power nonlinear
equation near each soliton will be the challenge for the linear theory. Another interesting point is that unlike the constructions in C\^ote-Mu\~noz \cite{CM}, C\^ote-Martel \cite{CM1}, Jendrej \cite{JJ1,JJ2}, Martel-Merle \cite{MM} which choose the initial data based on the Brouwer's fixed point theorem, in this paper, we construct the initial data for the unstable soliton case based on the Banach's fixed point theorem.

\smallskip

Returning to our model, the intuition is that for each potential, it will
trap some profile provided that $V_{j}$ has large negative part.
With the defocusing structure, the potentials and the nonlinearity
will produce stable solitons. They can also form excited solitons that
is the excited states to associated elliptic equations. In this paper,
we will construct the multisoliton structures with stable solitons and unstable solitons. Notice that one needs more delicate analysis in order to handle the unstable solitons in Section \ref{sec:Construction}.

\smallskip

Throughout the paper, we assume that
\begin{equation}
\left|V_{j}(x)\right|\lesssim\frac{1}{\left\langle 1+\left|x\right|\right\rangle ^{\beta}},\,\,\beta>3.\label{eq:decay}
\end{equation}
Before we formulate the main theorems, we recall Lorentz transformations
along the $x_{1}$ axis since one can deform a rotation to reduce
the general cases to this specific one. More precisely, for the moving frame, $(x-\vec{v}t)$, there is a unique rotation
$\rho_{v}$ so that after rotating, in the new frame $\text{\ensuremath{\left(z_{1},z_{2},z_{3}\right)}}$,
the vector $\vec{v}$ is along $z_{1}$, i.e. the moving frame becomes
$\left(z_{1}-\left|\vec{v}\right|\vec{e}_{1}t,z_{2},z_{3}\right).$ Then
we apply the Lorentz transformation along $\vec{e}_{1}$ with velocity
$v$. Define
\begin{equation}
\Lambda_{v}:=\left(\begin{array}{cccc}
\gamma & -\left|\vec{v}\right|\gamma & 0 & 0\\
-\left|\vec{v}\right|\gamma & \left|\vec{v}\right|\gamma & 0 & 0\\
0 & 0 & 1 & 0\\
0 & 0 & 0 & 1
\end{array}\right)\label{eq:Lmatrix}
\end{equation}
where 
\begin{equation}
\gamma=\frac{1}{\sqrt{1-\left|\vec{v}\right|^{2}}}.\label{eq:gamma}
\end{equation}
Then we consider the the following change of variables:
\begin{equation}
\left(\begin{array}{c}
t'\\
x_{1}'\\
x_{2}'\\
x_{3}'
\end{array}\right)=\Lambda_{v}\left(\begin{array}{cccc}
1 & 0 & 0 & 0\\
0\\
0 &  & \rho_{v}\\
0
\end{array}\right)\left(\begin{array}{c}
t\\
x_{1}\\
x_{2}\\
x_{3}
\end{array}\right).\label{eq:coord}
\end{equation}
In our model, applying the above  transformation, in the new frame $\left(x',t'\right),$ the moving potential $V_{j}\left(x-\vec{v}_{j}t\right)$
becomes 
\begin{equation}
V_{j}\left(\sqrt{1-\left|\vec{v}_{j}\right|^{2}}x_{1}',x_{2}',x_{3}'\right)=:V_{j}\left(m_{v_{j}}x'\right)\label{eq:potential}
\end{equation}
Setting
\begin{equation}
V_{j}^{v_{j}}(x')=V_{j}\left(m_{v_{j}}x'\right),\label{eq:Dpotential}
\end{equation}
then we consider the Schr\"odinger operator 
\begin{equation}
-\Delta+V_{j}^{v_{j}}(x)\label{eq:schop}
\end{equation}
 Let $W_{j}^{v_{j}}$ be the stable static state to
\begin{equation}
-\Delta W_{j}^{v_{j}}+V_{j}^{v_{j}}\left(x\right)W_{j}^{v_{j}}+\left(W_{j}^{v_{j}}\right)^{5}=0.\label{eq:ellipj}
\end{equation}
Note that
\[
W_{j}^{v_{j}}\left(x'\right)=W_{j}^{v_{j}}\left(m_{v_{j}}^{-1}\rho_{v_{j}}\left(x-v_{j}t\right)\right).
\]
By a stable state, we mean that the linearized operator
\begin{equation}
-\Delta+V_{j}^{v_{j}}\left(x\right)+5\left(W_{j}^{v_{j}}\right)^{4}\label{eq:Lj}
\end{equation}
has no eigenvalues nor zero resonance. For detailed definitions, see
Section \ref{sec:Preliminaries} and the Appendix on the linear theory.

Set
\begin{equation}
W_{j}(x)=W_{j}^{v_{j}}\left(m_{v_{j}}^{-1}\rho_{v_{j}}\left(x\right)\right).\label{eq:bubble}
\end{equation}
It is crucial to notice that 
\begin{equation}
\left|W_{j}(x)\right|\simeq\frac{1}{\left\langle x\right\rangle }\label{eq:bdecay}
\end{equation}
which causes the interactions among different solitons in our construction
are very strong. For more detailed discussions on the existence and
decay estimates, see Section \ref{sec:Preliminaries}.

\smallskip

We also need the Hamiltonian structure of wave equations to discuss
scattering. In general, we can write a general wave equation as
\begin{equation}
\partial_{tt}u-\Delta u=F(u,t)
\end{equation}
with initial data
\begin{equation}
u(x,0)=f(x),\,u_{t}(x,0)=g(x).
\end{equation}
Also consider the homogeneous free wave equation,
\begin{equation}
\partial_{tt}u_{0}-\Delta u_{0}=0
\end{equation}
with initial data
\begin{equation}
u_{0}(x,0)=f_{0}(x),\,\left(u_{0}\right)_{t}(x,0)=g_{0}(x).
\end{equation}
We reformulate the wave equation as a Hamiltonian system, 
\begin{equation}
\partial_{t}\left(\begin{array}{c}
u\\
\partial_{t}u
\end{array}\right)-\left(\begin{array}{cc}
0 & 1\\
-1 & 0
\end{array}\right)\left(\begin{array}{cc}
-\Delta & 0\\
0 & 1
\end{array}\right)\left(\begin{array}{c}
u\\
\partial_{t}u
\end{array}\right)=\left(\begin{array}{c}
0\\
F(u)
\end{array}\right).
\end{equation}
Setting 
\begin{equation}
U:=\left(\begin{array}{c}
u\\
\partial_{t}u
\end{array}\right),\,J:=\left(\begin{array}{cc}
0 & 1\\
-1 & 0
\end{array}\right),\,H_{F}:=\left(\begin{array}{cc}
-\Delta & 0\\
0 & 1
\end{array}\right)\,\text{and }F(U):=\left(\begin{array}{c}
0\\
F(u,t)
\end{array}\right),\label{eq:bigU}
\end{equation}
we can rewrite the free wave equation as 
\begin{equation}
\dot{U}_{0}-JH_{F}U_{0}=0,
\end{equation}
\begin{equation}
U_{0}[0]=\left(\begin{array}{c}
f_{0}\\
g_{0}
\end{array}\right)
\end{equation}
and the nonlinear wave equation as 
\begin{equation}
\dot{U}-JH_{F}U=F(U),
\end{equation}
\begin{equation}
U[0]=\left(\begin{array}{c}
f\\
g
\end{array}\right).
\end{equation}
The solution of the free wave equation is given by 
\begin{equation}
U_{0}=e^{tJH_{F}}U_{0}[0].
\end{equation}
In the following, we write 
\begin{equation}
U[t]=\left(u,u_{t}\right)^{t},\,W[t]=\left(\sum_{j=1}^{m}W_{j}\left(x-\vec{v}_{j}t\right),\,\partial_{t}\sum_{j=1}^{m}W_{j}\left(x-\vec{v}_{j}t\right)\right)^{t}.
\end{equation}
With the preparations and notations above, we can formulate our main
theorems with stable solitons:
\begin{thm}[Existence of purely multi-soliton solutions]
\label{thm:existence}
In $\mathbb{R}^3$, there exists a solution $u$ to 
\begin{equation}
\partial_{tt}u-\Delta u+\sum_{j=1}^{m}V_{j}\left(x-\vec{v}_{j}t\right)u+u^{5}=0,
\end{equation}
such that 
\begin{equation}
\lim_{t\rightarrow\infty}\left\Vert U[t]-W[t]\right\Vert _{\dot{H}^{1}\times L^{2}} = 0.
\end{equation}
Moreover, we have the decay rate 
\begin{equation}
\left\Vert U[t]-W[t]\right\Vert _{\dot{H}^{1}\times L^{2}}\lesssim\frac{1}{\sqrt{t}}.
\end{equation}
\textup{as $t\rightarrow\infty$.}
\end{thm}
Next we have the asymptotic stability of the multisoliton structure.
\begin{thm}[Asymptotic stability of the multisoliton]
\label{thm:Stability} Suppose that $0<\epsilon\ll1$ is
small enough and $1\ll t_{0}$ is large enough. Let $u$ solve
\begin{equation}
\partial_{tt}u-\Delta u+\sum_{j=1}^{m}V_{j}\left(x-\vec{v}_{j}t\right)u+u^{5}=0,\,\,\,x\in\mathbb{R}^{3}.
\end{equation}
Suppose at $t=t_{0}$,
\begin{equation}
\left\Vert U[t_{0}]-W[t_{0}]\right\Vert _{\dot{H}^{1}\times L^{2}}\leq\epsilon.
\end{equation}
Then there exists free data
\begin{equation}
U_{0}[0]=\left(f_{0},g_{0}\right)^{t}\in\dot{H}^{1}\times L^{2}
\end{equation}
such that
\begin{equation}
\left\Vert U[t]-W[t]-e^{tJH_{F}}U_{0}[0]\right\Vert _{\dot{H}^{1}\times L^{2}}\rightarrow0.\,\,t\rightarrow\infty.
\end{equation}
In other words, the error $u(t)-\sum_{j=1}^{m}W_{j}\left(x-\vec{v}_{j}t\right)$\textup{
}scatters to a free wave.
\end{thm}
Here we briefly discuss strong interactions among these solitons. For
simplicity, we consider the case when $m=2$ as in Section \ref{sec:Construction}.
Around near two solitons, we define
\begin{equation}
h(t):=u(t)-W_{1}\left(x\right)-W_{2}\left(x-\vec{v}t\right).\label{eq:eqh-1}
\end{equation}
We consider the equation for $h$. Plugging everything in the equation,
we have 
\begin{align}
\partial_{tt}h-\Delta h+h^{5}+\left(V_{1}(x)+5W_{1}^{4}(x)\right)h+\left(V_{2}\left(x-\vec{v}t\right)+5W_{2}^{4}\left(x-\vec{v}t\right)\right)h+a(x,t)h\nonumber \\
=F_{1}(x,t)+F_{2}(x,t)+F(x,t)+N(h,x,t)
\end{align}
with 
\begin{equation}
a(x,t)=20W_{1}^{3}(x)W_{2}\left(x-\vec{v}t\right)+30W_{1}^{2}(x)W_{2}^{2}\left(x-\vec{v}t\right)+20W_{1}(x)W_{2}^{3}\left(x-\vec{v}t\right)
\end{equation}
\begin{equation}
F_{1}(x,t)=5W_{1}^{4}(x)W_{2}\left(x-\vec{v}t\right)+V_{1}(x)W_{2}(x-\vec{v}t)
\end{equation}
\begin{equation}
F_{2}(x,t)=5W_{1}(x)W_{2}^{4}\left(x-\vec{v}t\right)+W_{1}(x)V_{2}(x-\vec{v}t)
\end{equation}
\begin{equation}
F(x,t)=10W_{1}^{3}(x)W_{2}^{2}\left(x-\vec{v}t\right)+10W_{1}^{2}(x)W_{2}^{3}\left(x-\vec{v}t\right)
\end{equation}
and $N(h,x,t)$ is  quadratic or higher in $h$. For more details, see Section
\ref{sec:Construction}. We notice that $F(x,t)$ is easy to handle
but $F_{1},\,F_{2}\notin L_{t}^{1}L_{x}^{2}$ . One can not simply
apply the energy estimate and Strichartz estimates directly. $F_{1}$
and $F_{2}$ precisely show that due to the slow decay rate of the
solitons, see Section \ref{sec:Preliminaries}, some terms in the nonlinear
interactions decay slowly. To overcome these terms, we need the local
energy decay and reversed Strichartz estimates with inhomogeneous
terms in the reversed norm. Moreover, due to the failure of the endpoint
Strichartz estimate in $\mathbb{R}^{3}$, to handle the quadratic
term of $h$ in the nonlinearity, one also needs the endpoint reversed
Strichartz estimate and revered type local decay estimates. It is also novel in the nonradial setting.

\smallskip

All the above results can be extended to the multisoliton construction with unstable excited solitons, see Section \ref{sec:Construction}. The linear model still plays a pivotal role. It is interesting to compare our method which is based on linear estimates with the constructions of multisolitons by nonlinear techniques developed in for example, in Martel \cite{Mart}, Merle \cite{Merle}, C\^ote-Mu\~noz \cite{CM}, C\^ote-Martel \cite{CM1}, Jendrej \cite{JJ1,JJ2}, Martel-Merle \cite{MM}. First of all, our linear model can be used to analyze the stability of the multisoliton structure. Secondly, the scattering state we construct in this paper is based on the Banach's fixed point theorem other than the Brouwer's fixed-point theorem. On the other hand, when we need to deal with the purely-soliton solution with unstable solitons, we also need the weak convergence technique which is commonly used in the nonlinear method.
\subsection*{Notation}

\textquotedblleft $A:=B\lyxmathsym{\textquotedblright}$ or $\lyxmathsym{\textquotedblleft}B=:A\lyxmathsym{\textquotedblright}$
is the definition of $A$ by means of the expression $B$. We use
the notation $\langle x\rangle=\left(1+|x|^{2}\right)^{\frac{1}{2}}$.
The bracket $\left\langle \cdot,\cdot\right\rangle $ denotes the
distributional pairing and the scalar product in the spaces $L^{2}$,
$L^{2}\times L^{2}$ . For positive quantities $a$ and $b$, we write
$a\lesssim b$ for $a\leq Cb$ where $C$ is some prescribed constant.
Also $a\simeq b$ for $a\lesssim b$ and $b\lesssim a$. 

\subsection*{Organization}

The paper is organized as follows: In Section \ref{sec:Preliminaries},
we list some existence and decay results on the solutions to elliptic
equations. In Section \ref{sec:Construction}, we establish the 
main theorems in this paper. The constructions with unstable solitons
will be shown. Finally, in the Appendix, we briefly recall the
linear theory that we need in this paper based on results from Chen
\cite{GC3,GC2}. We also discuss the scattering behavior of the nonlinear equation.

\subsection*{Acknowledgment}

I feel deeply grateful to my advisor Professor Wilhelm Schlag for
his kind encouragement, discussions, comments and all the support.
I also want to thank Jacek Jendrej for many useful and enlightening
discussions.

\section{Preliminaries: Static States\label{sec:Preliminaries}}

In order to construct the multisoliton solution to the equation 
\begin{equation}
\partial_{tt}u-\Delta u+\sum_{j=1}^{m}V_{j}\left(x-\vec{v}_{j}t\right)u+u^{5}=0,
\end{equation}
we first have to understand the soliton trapped by each potential separately. 

Performing a Lorentz transformation, it suffices to understand the
model elliptic equation:
\begin{equation}
-\Delta u+Vu+u^{5}=0.\label{eq:static}
\end{equation}

\begin{defn}[Stability of a static solution] A solution $W$ to the equation \eqref{eq:static}
is called a stable solution if the linearized opearator 
\begin{equation}
L_{W}=-\Delta+V+5W^{4}\label{eq:L}
\end{equation}
has no eigenvalues nor zero resonance.
\end{defn}
We follow Jia-Liu-Xu \cite{JLX} and Jia-Liu-Schlag-Xu \cite{JLSX}.
Define the energy functional 
\begin{equation}
J(u):=\int_{R^{3}}\frac{|\nabla u|^{2}}{2}+\frac{Vu^{2}}{2}+\frac{u^{6}}{6}(x,t)\,dx.\label{eq:J}
\end{equation}
In general when the negative part $V$ of the potential is large,
one can expect that there is a unique positive ground state, which
is the global minimizer of energy functional and has negative energy.
In addition, there can be a number of \textquotedblleft excited states\textquotedblright{}
with higher energies (see Appendix A of  \cite{JLX} for more details).
It is well known the ground state is asymptotically stable at least
when $V$ decays fast. However the dynamics around the excited states
can be very complicated even in perturbative regime (even with radial
data), involving stable and unstable manifolds. It arises some difficulties
to take these unstable excited states as the solitons in our construction.

Here we list some important results regarding the elliptic equation
from \cite{JLX,JLSX}.
\begin{lem}
Consider $J$ as a functional defined in $\dot{H}^{1}(\text{\ensuremath{\mathbb{R}}}^{3})$.
If the operator $-\Delta-V$ has negative eigenvalues then there exists
a global minimizer $Q>0$ with $J(Q)<0$. If $-\Delta-V$ has no negative
eigenvalues, then the only steady state solution $u\in\dot{H}^{1}(\mathbb{R}^{3})$
to equation \eqref{eq:static} is $u\equiv0$. 
\end{lem}
\begin{thm}
Fix $\beta>2$. Define 
\begin{equation}
Y:=\{V\in C(\mathbb{R}^{3}):\,V\,\,{\rm is\,\,radial\,\,and\,\,}\sup_{x}(1+|x|)^{\beta}|V(x)|<\infty\}.
\end{equation}
 For $V$ in a dense open set $\Omega\subset Y$, there are only finitely
many radial steady states to equation \eqref{eq:static}. 
\end{thm}
\begin{thm}

Let $V\in Y$. For any $c\in\mathbb{R}$, there exists a unique radial
solution $u_{c}\in\dot{H}^{1}(B_{r}^{c})$ for any $r>0$ to 
\begin{equation}
-\Delta u+V(x)u+u^{5}=0,\,\,{\rm in}\,\,R^{3}\backslash\{0\},\label{eq:elliptic}
\end{equation}
 with 
\begin{equation}
\left|u(x)-\frac{c}{|x|}\right|=o(\frac{1}{|x|}),\,\,{\rm as}\,\,|x|\to\infty.
\end{equation}
 If $u_{c}\in\dot{H}^{1}(\mathbb{R}^{3})$, then $u_{c}\in C^{1}(\mathbb{R}^{3})$
and
\begin{equation}
-\Delta u_{c}+Vu_{c}+u_{c}^{5}=0\,\,{\rm in}\,\,R^{3}.
\end{equation}
\end{thm}
If we take the ground states as the solitons, we notice that the
optimal decay rate is $\frac{1}{\left\langle x\right\rangle }$. Even
if one can assume that $V(x)$ decays very fast, there is no hope to
improve the decay rate for the ground state.
\begin{lem}
Let $R,\,\beta_{1}$ be sufficiently large. There exist $\epsilon,\,\delta>0$,
such that if 
\begin{equation}
\sup_{x\in B_{R}^{c}}|x|^{\beta_{1}}|b(x)|<\delta,
\end{equation}
 then any solution $u$ to 
\begin{equation}
-\Delta u+b(x)u+u^{5}=0,\,\,x\in B_{R}^{c}
\end{equation}
 with $u|_{\partial B_{R}}\ge1$ satisfies 
\begin{equation}
|u(x)|\ge\frac{\epsilon}{|x|},\,\,{\rm for}\,\,x\in B_{R}^{c}.
\end{equation}
\end{lem}
Naively one might expect excited states to be unstable, since they
change sign. However in general this may not be the case, as seen
from the following theorem.
\begin{thm}
There exists an open set $\mathcal{O}\in Y$ such that for any $V\in\mathcal{O}$,
there exists an excited state $\phi$ to equation \eqref{eq:static}
which is stable. 
\end{thm}

\section{The Construction and Stability of Multisolitons \label{sec:Construction}}

In this section, we prove the main results of this paper. For simplicity,
we discuss the case when $m=2$:
\begin{equation}
\partial_{tt}u-\Delta u+V_{1}(x)u+V_{2}\left(x-\vec{v}t\right)u+u^{5}=0\label{eq:eqtwo}
\end{equation}
and $\vec{v}$  is along the $x_{1}$ direction.

We start with an energy estimate based on the local energy decay for
the free wave equation, c.f. the Appendices of \cite{GC2}. This lemma
is particularly useful to handle strong interaction terms, see Remark
\ref{rem:eg}.
\begin{lem}
\label{lem:energyL}Consider
\begin{equation}
\partial_{tt}u-\Delta u=H
\end{equation}
with initial data
\begin{equation}
u(0)=f,\,u_{t}(0)=g.
\end{equation}
Then for any $\epsilon>0$ and  $|\vec{v}|<1$, one has
\begin{equation}\label{eq:decay1}
\sup_{t\geq0}\left(\|\nabla u(t)\|_{L^{2}}+\|u_{t}(t)\|_{L^{2}}\right)\lesssim_{\epsilon}\|g\|_{L^{2}}+\|f\|_{\dot{H}^{1}}+\left\Vert \left\langle x\right\rangle ^{\frac{1}{2}+\epsilon}H(t)\right\Vert _{L_{t,x}^{2}}
\end{equation}and 
\begin{equation}\label{eq:decay2}
\sup_{t\geq0}\left(\|\nabla u(t)\|_{L^{2}}+\|u_{t}(t)\|_{L^{2}}\right)\lesssim_{\epsilon}\|g\|_{L^{2}}+\|f\|_{\dot{H}^{1}}+\left\Vert \left\langle x-\vec{v}t\right\rangle ^{\frac{1}{2}+\epsilon}H(t)\right\Vert _{L_{t,x}^{2}}
\end{equation}
\end{lem}
\begin{proof}
we set $A=\sqrt{-\Delta}$ and notice that 
\begin{equation}
\left\Vert Af\right\Vert _{L^{2}}\simeq\left\Vert f\right\Vert _{\dot{H}^{1}},\,\,\forall f\in C^{\infty}\left(\mathbb{R}^{3}\right).
\end{equation}
For real-valued $u=\left(u_{1},u_{2}\right)\in\mathcal{H}=\dot{H}^{1}\left(\mathbb{R}^{3}\right)\times L^{2}\left(\mathbb{\mathbb{R}}^{3}\right)$,
we write 
\begin{equation}
U:=Au_{1}+iu_{2}.
\end{equation}
and then
\begin{equation}
\left\Vert U\right\Vert _{L^{2}}\simeq\left\Vert \left(u_{1},u_{2}\right)\right\Vert _{\mathcal{H}}.
\end{equation}
We also notice that $u$ solves the original equation if and only
if 
\begin{equation}
U:=Au+i\partial_{t}u
\end{equation}
satisfies 
\begin{equation}
i\partial_{t}U=AU+H,
\end{equation}
\begin{equation}
U(0)=Af+ig\in L^{2}\left(\mathbb{R}^{3}\right).
\end{equation}
By Duhamel's formula, 
\begin{equation}
U(t)=e^{itA}U(0)-i\int_{0}^{t}e^{-i\left(t-s\right)A}H(s)\,ds.
\end{equation}
We will only prove the first estimate \eqref{eq:decay1}. The second one \eqref{eq:decay2} follows the same way with the the standard local energy decay replaced by the local energy decay developed in \cite{GC3,GC2}.
 
From the energy estimate for the free evolution,
\begin{equation}
\sup_{t\in\mathbb{R}}\left\Vert e^{itA}U(0)\right\Vert _{L_{x}^{2}}\lesssim\left\Vert U(0)\right\Vert _{L^{2}}.
\end{equation}
It suffices to bound 
\begin{equation}
\sup_{t\in\mathbb{R}}\left\Vert \int_{0}^{\infty}e^{-i\left(t-s\right)A}\left\langle x\right\rangle ^{-\frac{1}{2}-\epsilon}\left(\left\langle x\right\rangle ^{\frac{1}{2}+\epsilon}H(s)\right)\,ds\right\Vert _{L_{x}^{2}}.
\end{equation}
Denote 
\begin{equation}
\left\langle x\right\rangle ^{\frac{1}{2}+\epsilon}H(s)=N.
\end{equation}
It is clear that 
\begin{equation}
\left\Vert \int_{0}^{\infty}e^{-i\left(t-s\right)A}\left\langle x\right\rangle ^{-\frac{1}{2}-\epsilon}N(s)\,ds\right\Vert _{L_{t}^{\infty}L_{x}^{2}}\leq\left\Vert \widetilde{K}\right\Vert _{L_{t}^{2}L_{x}^{2}\rightarrow L_{t}^{\infty}L_{x}^{2}}\left\Vert N\right\Vert _{L_{t}^{2}L_{x}^{2}},
\end{equation}
where 
\begin{equation}
\left(\widetilde{K}F\right)(t):=\int_{0}^{\infty}e^{-i\left(t-s\right)A}\left\langle x\right\rangle ^{-\frac{1}{2}-\epsilon}F(s)\,ds.
\end{equation}
We need to estimate
\begin{equation}
\left\Vert \widetilde{K}\right\Vert _{L_{t}^{2}L_{x}^{2}\rightarrow L_{t}^{\infty}L_{x}^{2}}.
\end{equation}
Testing against $F\in L_{t}^{2}L_{x}^{2}$, clearly, 
\begin{equation}
\left\Vert \widetilde{K}F\right\Vert _{L_{t}^{\infty}L_{x}^{2}}\leq\left\Vert e^{-itA}\right\Vert _{L^{2}\rightarrow L_{t}^{\infty}L_{x}^{2}}\left\Vert \int_{0}^{\infty}e^{isA}\left\langle x\right\rangle ^{-\frac{1}{2}-\epsilon}F(s)\,ds\right\Vert _{L^{2}}.\label{eq:TKF-1-3}
\end{equation}
The first factors on the right-hand side of \eqref{eq:TKF-1-3} is bounded
by the energy estimate for the free evolution. Consider the second
factor, by duality, it suffices to show 
\begin{equation}
\left\Vert \left\langle x\right\rangle ^{-\frac{1}{2}-\epsilon}e^{-itA}\phi\right\Vert _{L_{t}^{2}L_{x}^{2}}\lesssim\left\Vert \phi\right\Vert _{L^{2}},\,\forall\phi\in L^{2}\left(\mathbb{R}^{3}\right),
\end{equation}
which is local energy decay. 

For estimate \eqref{eq:decay2}, we apply 
\begin{equation}
\left\Vert \left\langle x-\vec{v}t\right\rangle ^{-\frac{1}{2}-\epsilon}e^{-itA}\phi\right\Vert _{L_{t}^{2}L_{x}^{2}}\lesssim\left\Vert \phi\right\Vert _{L^{2}},\,\forall\phi\in L^{2}\left(\mathbb{R}^{3}\right)
\end{equation}
in the appendices in \cite{GC2} or Corollary 2.10 in \cite{GC3}.

Hence 
\begin{equation}
\left\Vert \int_{0}^{\infty}e^{isA}\left\langle x\right\rangle ^{-\frac{1}{2}-\epsilon}F(s)\,ds\right\Vert _{L^{2}}\lesssim\left\Vert F\right\Vert _{L_{t}^{2}L_{x}^{2}}.
\end{equation}
Therefore, indeed, we have 
\begin{equation}
\left\Vert \widetilde{K}\right\Vert _{L_{t}^{2}L_{x}^{2}\rightarrow L_{t}^{\infty}L_{x}^{2}}\leq C
\end{equation}
and 
\begin{equation}
\sup_{t\in\mathbb{R}}\left\Vert \int_{0}^{\infty}e^{-i\left(t-s\right)A}\left\langle x\right\rangle ^{-\frac{1}{2}-\epsilon}N(s)\,ds\right\Vert _{L_{x}^{2}}\lesssim\left\Vert N\right\Vert _{L_{t,x}^{2}}.
\end{equation}
Therefore, we have 
\begin{equation}
\sup_{t\geq0}\left(\|\nabla u(t)\|_{L^{2}}+\|u_{t}(t)\|_{L^{2}}\right)\lesssim\|g\|_{L^{2}}+\|f\|_{\dot{H}^{1}}+\left\Vert \left\langle x\right\rangle ^{\frac{1}{2}+\epsilon}H(t)\right\Vert _{L_{t,x}^{2}}
\end{equation}
and
\begin{equation}
\sup_{t\geq0}\left(\|\nabla u(t)\|_{L^{2}}+\|u_{t}(t)\|_{L^{2}}\right)\lesssim\|g\|_{L^{2}}+\|f\|_{\dot{H}^{1}}+\left\Vert \left\langle x-\vec{v}t\right\rangle ^{\frac{1}{2}+\epsilon}H(t)\right\Vert _{L_{t,x}^{2}}
\end{equation}
as claimed
\end{proof}
\begin{rem}
\label{rem:eg}As a concrete example, we set 
\begin{equation}
H(t)=\frac{1}{\left\langle x\right\rangle ^{4}}\frac{1}{\left\langle x-\vec{v}t\right\rangle }.\label{eq:exH}
\end{equation}which is one of the interaction terms in the nonlinear model.

We want to solve 
\begin{equation}
\partial_{tt}u-\Delta u=H
\end{equation}
Taking $1<\eta\ll\left|\vec{v}\right|$, consider
\begin{equation}
\int_{\mathbb{R}^{3}}\frac{\left\langle x\right\rangle ^{1+2\epsilon}}{\left\langle x\right\rangle ^{8}}\frac{1}{\left\langle x-\vec{v}t\right\rangle ^{2}}\,dx.
\end{equation}
Splitting integral into three pieces:
\begin{equation}
\int_{\left|x\right|\leq\eta t}\frac{\left\langle x\right\rangle ^{1+2\epsilon}}{\left\langle x\right\rangle ^{8}}\frac{1}{\left\langle x-\vec{v}t\right\rangle ^{2}}\,dx\lesssim\frac{1}{\left\langle t\right\rangle ^{2}}
\end{equation}
\begin{equation}
\int_{\left|x-\vec{v}t\right|\leq\eta t}\frac{\left\langle x\right\rangle ^{1+2\epsilon}}{\left\langle x\right\rangle ^{8}}\frac{1}{\left\langle x-\vec{v}t\right\rangle ^{2}}\,dx\lesssim\frac{1}{\left\langle t\right\rangle ^{6-2-2\epsilon-1}}\lesssim\frac{1}{\left\langle t\right\rangle ^{2}}
\end{equation}
\begin{equation}
\int_{\left|x-\vec{v}t\right|\geq\eta t,\,\left|x\right|\geq\eta t}\frac{\left\langle x\right\rangle ^{1+2\epsilon}}{\left\langle x\right\rangle ^{8}}\frac{1}{\left\langle x-\vec{v}t\right\rangle ^{2}}\,dx\lesssim\frac{1}{\left\langle t\right\rangle ^{2}}.
\end{equation}
Therefore 
\begin{equation}
\int_{\mathbb{R}^{3}}\frac{\left\langle x\right\rangle ^{2+2\epsilon}}{\left\langle x\right\rangle ^{8}}\frac{1}{\left\langle x-\vec{v}t\right\rangle ^{2}}\,dx\lesssim\frac{1}{\left\langle t\right\rangle ^{2}}.
\end{equation}
Clearly 
\begin{equation}
\left\Vert \left\langle x\right\rangle ^{\frac{1}{2}+\epsilon}H(t)\right\Vert _{L_{t,x}^{2}}\lesssim\left(\int_{0}^{\infty}\frac{1}{\left\langle t\right\rangle ^{2}}\right)^{\frac{1}{2}}<\infty.
\end{equation}
If we consider the case that 
\begin{equation}
\left\Vert u(t)\right\Vert _{\dot{H}^{1}\times L^{2}}\rightarrow0,\,\,t\rightarrow\infty,
\end{equation}
then
\begin{equation}
\left\Vert u(t_{1})\right\Vert _{\dot{H}^{1}\times L^{2}}=\left(\int_{t_{1}}^{\infty}\left(\int_{\mathbb{R}^{3}}\frac{\left\langle x\right\rangle ^{2+2\epsilon}}{\left\langle x\right\rangle ^{8}}\frac{1}{\left\langle x-\vec{v}t\right\rangle ^{2}}\,dx\right)dt\right)^{\frac{1}{2}}\lesssim\frac{1}{\sqrt{t_{1}}}.
\end{equation}
We point out that for this $H(t)$ as the inhomogeneous term, the trivial
energy estimate fails since $H(t)\notin L_{t}^{1}L_{x}^{2}$.
\end{rem}

\subsection{Stable solitons}

Later on, throughout this section, we will use the short-hand notation:
\begin{equation}
L_{t}^{p}L_{x}^{q}:=L_{t}^{p}\left([t_{0},\infty),\,L_{x}^{q}\right).\label{eq:notation}
\end{equation}
where $t_{0}$ is the large time which only depends on prescribed constants from Theorem \ref{thm:Stability}.

We first prove Theorem \ref{thm:Stability}. Setting 
\begin{equation}
h(t):=u(t)-W_{1}\left(x\right)-W_{2}\left(x-\vec{v}t\right),\label{eq:eqh}
\end{equation}
it is well-known that we just need to show that $h$ is bounded in
Strichartz norms, see Chen \cite{GC3} or Theorem \ref{thm:scattering} in the Appendix.
\begin{proof}[Proof of Theorem \ref{thm:Stability}]
By construction, we have 
\begin{align}
	\partial_{tt}h-\Delta h+h^{5}+\left(V_{1}(x)+5W_{1}^{4}(x)\right)h+\left(V_{2}\left(x-\vec{v}t\right)+5W_{2}^{4}\left(x-\vec{v}t\right)\right)h+a(x,t)h\nonumber \\
	=F_{1}(x,t)+F_{2}(x,t)+F(x,t)+N(h,x,t)
\end{align}
with 
\begin{equation}
a(x,t):=20W_{1}^{3}(x)W_{2}\left(x-\vec{v}t\right)+30W_{1}^{2}(x)W_{2}^{2}\left(x-\vec{v}t\right)+20W_{1}(x)W_{2}^{3}\left(x-\vec{v}t\right)
\end{equation}
\begin{equation}
F_{1}(x,t):=5W_{1}^{4}(x)W_{2}\left(x-\vec{v}t\right)+V_{1}(x)W_{2}(x-\vec{v}t)
\end{equation}
\begin{equation}
F_{2}(x,t):=5W_{1}(x)W_{2}^{4}\left(x-\vec{v}t\right)+W_{1}(x)V_{2}(x-\vec{v}t)
\end{equation}
\begin{equation}
F(x,t):=10W_{1}^{3}(x)W_{2}^{2}\left(x-\vec{v}t\right)+10W_{1}^{2}(x)W_{2}^{3}\left(x-\vec{v}t\right)
\end{equation}

\begin{align}
	N(h,x,t) & :=\left(10W_{1}^{3}(x)+30W_{1}^{2}(x)W_{2}\left(x-\vec{v}t\right)+30W_{1}(x)W_{2}^{2}\left(x-\vec{v}t\right)+10W_{2}^{3}(x-\vec{v}t)\right)h^{2}\nonumber \\
	& +\left(10W_{1}^{2}(x)+3W_{1}(x)W_{2}\left(x-\vec{v}t\right)+10W_{2}^{2}\left(x-\vec{v}t\right)\right)h^{3}\\
	& +\left(5W_{1}(x)+5W_{2}\left(x-\vec{v}t\right)\right)h^{4}\nonumber 
\end{align}
Furthermore, we denote 
\begin{equation}
M_{1}(h,x,t):=M_{1,1}(h,x,t)+M_{1,2}(h,x,t)+M_{1,3}(h,x,t)+M_{1,4}(h,x,t)
\end{equation}
where
\begin{equation}
M_{1,1}(h,x,t):=10W_{1}^{3}(x)h^{2},\,M_{1,2}(h,x,t):=30W_{1}^{2}(x)W_{2}\left(x-\vec{v}t\right)h^{2}
\end{equation}
and
\begin{equation}
M_{1,3}(h,x,t):=30W_{1}(x)W_{2}^{2}\left(x-\vec{v}t\right)h^{2},\,M_{1,4}(h,x,t):=10W_{2}^{3}(x-\vec{v}t)h^{2}.
\end{equation}
Also we use the notation:
\begin{equation}
M_{2}(h,x,t):=\left(10W_{1}^{2}(x)+3W_{1}(x)W_{2}\left(x-\vec{v}t\right)+10W_{2}^{2}\left(x-\vec{v}t\right)\right)h^{3},
\end{equation}
\begin{equation}
M_{3}(h,x,t):=\left(5W_{1}(x)+5W_{2}\left(x-\vec{v}t\right)\right)h^{4}
\end{equation}
Consider the iteration scheme
\begin{align}
	\partial_{tt}h_{i+1}-\Delta h_{i+1}+\left(V_{1}(x)+5W_{1}^{4}(x)\right)h_{i+1}+\left(V_{2}\left(x-\vec{v}t\right)+5W_{2}^{4}\left(x-\vec{v}t\right)\right)h_{i+1}\nonumber \\
	=F_{1}(x,t)+F_{2}(x,t)+F(x,t)+N(h_{i},x,t)-a(x,t)h_{i}-h_{i}^{5}.
\end{align}
Define the Strichartz norm of $h$ as 
\begin{equation}
\left\Vert h\right\Vert _{Stri}:=\sup_{\frac{1}{2}=\frac{1}{p}+\frac{3}{q},\,p=3,\,4,\,5}\left\Vert h\right\Vert _{L_{t}^{p}\left([t_{0},\infty),\,L_{x}^{q}\right).}\label{eq:StriNorm}
\end{equation}
For a function $G(x,t)$, we use the notation: 
\begin{equation}
G^{S}(x,t):=G(x+\vec{v}t,t).
\end{equation}
As in the Appendix, for fixed $\epsilon_{0}>0$ small, define the strong interactions spaces $I$ as
\begin{equation}
I:=\left\{ G(x,t)\in L_{x}^{\frac{3}{2},1}L_{t}^{2}\bigcap L_{x_{1}}^{1}L_{\widehat{x_{1}}}^{2,1}L_{t}^{2},\,\text{such that}\,\,\left\Vert \left\langle x\right\rangle ^{\frac{1}{2}+\epsilon_{0}}G \right\Vert _{L_{t,x}^{2}}<\infty\right\} ,\label{eq:Ispace-1}
\end{equation}
and local decay space $D$ as
\begin{equation}
D:=\left\{ \left\langle x\right\rangle ^{-3}G(x,t)\in L_{x}^{\frac{3}{2},1}L_{t}^{\infty}\bigcap L_{x_{1}}^{1}L_{\widehat{x_{1}}}^{2,1}L_{t}^{\infty}\right\} .\label{eq:Dspace}
\end{equation}
Define 
\begin{equation}
\left\Vert G\right\Vert _{I}:=\max\left\{ \,\left\Vert \left\langle x\right\rangle ^{\frac{1}{2}+\epsilon_{0}}G\right\Vert _{L_{t,x}^{2}},\,\left\Vert G\right\Vert _{L_{x_{1}}^{1}L_{\widehat{x_{1}}}^{2,1}L_{t}^{2}},\,\left\Vert G\right\Vert _{L_{x}^{\frac{3}{2},1}L_{t}^{2}}\right\} ,\label{eq:Inorm-1}
\end{equation}
and 
\begin{equation}
\left\Vert G\right\Vert _{D}:=\max\left\{ \left\Vert \left\langle x\right\rangle ^{-3}G(x,t)\right\Vert _{L_{x_{1}}^{1}L_{\widehat{x_{1}}}^{2,1}L_{t}^{\infty}},\,\left\Vert \left\langle x\right\rangle ^{-3}G\right\Vert _{L_{x}^{\frac{3}{2},1}L_{t}^{\infty}}\right\} .\label{eq:Dnorm}
\end{equation}
Using the notations before, by estimate \eqref{eq:RStrichartz} in Theorem
\ref{thm:EndRStriCB} from the linear theory in the Appendix, we have
\begin{align}
	\sup_{x\in\mathbb{R}^{3}}\left(\int_{t_{0}}^{\infty}\left|h_{i+1}(x,t)\right|^{2}dt\right)^{\frac{1}{2}} & \lesssim\|f\|_{L^{2}}+\|g\|_{\dot{H}^{1}}+\left\Vert F\right\Vert _{L_{t}^{1}L_{x}^{2}}\nonumber \\
	& +\left\Vert a(x,t)h_{i}\right\Vert _{L_{t}^{1}L_{x}^{2}}+\left\Vert h_{i}^{5}\right\Vert _{L_{t}^{1}L_{x}^{2}}\\
	& +\left\Vert M_{2}(h_{i},x,t)\right\Vert _{L_{t}^{1}L_{x}^{2}}+\left\Vert M_{3}(h_{i},x,t)\right\Vert _{L_{t}^{1}L_{x}^{2}}\nonumber \\
	& +\left\Vert M_{1,2}(h_{i},x,t)\right\Vert _{L_{t}^{1}L_{x}^{2}}+\left\Vert M_{1,3}(h_{i},x,t)\right\Vert _{L_{t}^{1}L_{x}^{2}}\nonumber \\
	& +\left\Vert M_{1,1}(h_{i},x,t)\right\Vert _{I}+\left\Vert M_{1,4}^{S}(h_{i},x,t)\right\Vert _{I}\nonumber \\
	& +\left\Vert F_{1}\right\Vert _{I}+\left\Vert F_{2}^{S}\right\Vert _{I}.\nonumber 
\end{align}
Applying H\"older's inequality and Strichartz estimates, we can estimate
\begin{equation}
\left\Vert a(x,t)h_{i}\right\Vert _{L_{t}^{1}L_{x}^{2}}\lesssim\left\Vert a\right\Vert _{L_{t}^{\frac{5}{4}}L_{x}^{\frac{5}{2}}}\left\Vert h_{i}\right\Vert _{L_{t}^{5}L_{x}^{10},}
\end{equation}
\begin{equation}
\left\Vert M_{2}(h_{i},x,t)\right\Vert _{L_{t}^{1}L_{x}^{2}}\lesssim\left\Vert h_{i}\right\Vert _{L_{t}^{3}L_{x}^{18},}
\end{equation}
\begin{equation}
\left\Vert M_{3}(h_{i},x,t)\right\Vert _{L_{t}^{1}L_{x}^{2}}\lesssim\left\Vert h_{i}\right\Vert _{L_{t}^{4}L_{x}^{12},}
\end{equation}
and 
\begin{equation}
\left\Vert M_{1,2}(h_{i},x,t)\right\Vert _{L_{t}^{1}L_{x}^{2}},\,\left\Vert M_{1,3}(h_{i},x,t)\right\Vert _{L_{t}^{1}L_{x}^{2}}\lesssim\left\Vert h_{i}\right\Vert _{L_{t}^{4}L_{x}^{12}.}
\end{equation}
For the strong-interaction terms, we notice that 
\begin{align}
	\left\Vert M_{1,1}(h_{i},x,t)\right\Vert _{L_{x}^{\frac{3}{2},1}L_{t}^{2}} & \lesssim\left\Vert W_{1}^{3}(x)h_{i}\right\Vert _{L_{x}^{\frac{3}{2},1}L_{t}^{\infty}}\left\Vert h_{i}\right\Vert _{L_{x}^{\infty}L_{t}^{2}}\nonumber \\
	& \lesssim\left\Vert h_{i}\right\Vert _{D}\left\Vert h_{i}\right\Vert _{L_{x}^{\infty}L_{t}^{2}}.
\end{align}
In the same manner, we can estimate all other norms in the definition
of $I$ and conclude that 
\begin{equation}
\left\Vert M_{1,1}(h_{i},x,t)\right\Vert _{I}\lesssim\left\Vert h_{i}\right\Vert _{D}\left\Vert h_{i}\right\Vert _{L_{x}^{\infty}L_{t}^{2}}.
\end{equation}
Similarly, 
\begin{equation}
\left\Vert M_{1,4}^{S}(h_{i},x,t)\right\Vert _{I}\lesssim\left\Vert h_{i}^{S}\right\Vert _{D}\left\Vert h_{i}^{S}\right\Vert _{L_{x}^{\infty}L_{t}^{2}}.
\end{equation}
Therefore, we know that 
\begin{align}
	\sup_{x\in\mathbb{R}^{3}}\left(\int_{t_{0}}^{\infty}\left|h_{i+1}(x,t)\right|^{2}dt\right)^{\frac{1}{2}}\lesssim & \|f\|_{L^{2}}+\|g\|_{\dot{H}^{1}}+\left\Vert F\right\Vert _{L_{t}^{1}L_{x}^{2}}\nonumber \\
	& +\left\Vert h_{i}\right\Vert _{D}\left\Vert h_{i}\right\Vert _{L_{x}^{\infty}L_{t}^{2}}+\left\Vert h_{i}^{S}\right\Vert _{D}\left\Vert h_{i}^{S}\right\Vert _{L_{x}^{\infty}L_{t}^{2}}\nonumber \\
	& +\left\Vert h_{i}\right\Vert _{Stri}+\left\Vert h_{i}\right\Vert _{Stri}^{5}+\left\Vert F_{1}\right\Vert _{I}+\left\Vert F_{2}^{S}\right\Vert _{I}.
\end{align}
Similarly, by estimate \eqref{eq:RStrichartzS} from in Theorem \ref{thm:EndRStriCB}, one has
\begin{align}
	\sup_{x\in\mathbb{R}^{3}}\left(\int_{t_{0}}^{\infty}\left|h_{i+1}(x+\vec{v}t,t)\right|^{2}dt\right)^{\frac{1}{2}}\lesssim & \|f\|_{L^{2}}+\|g\|_{\dot{H}^{1}}+\left\Vert F\right\Vert _{L_{t}^{1}L_{x}^{2}}\nonumber \\
	& +\left\Vert h_{i}\right\Vert _{D}\left\Vert h_{i}\right\Vert _{L_{x}^{\infty}L_{t}^{2}}+\left\Vert h_{i}^{S}\right\Vert _{D}\left\Vert h_{i}^{S}\right\Vert _{L_{x}^{\infty}L_{t}^{2}}\nonumber \\
	& +\left\Vert h_{i}\right\Vert _{Stri}+\left\Vert h_{i}\right\Vert _{Stri}^{5}+\left\Vert F_{1}\right\Vert _{I}+\left\Vert F_{2}^{S}\right\Vert _{I}.
\end{align}
For the local decay, by the estimates \eqref{eq:Localinfty} and \eqref{eq:localinftyS}
from Theorem \ref{thm:EndRStriCB}, we conclude that 
\begin{align}
	\left\Vert h_{i+1}(x,t)\right\Vert _{D} & \lesssim\|g\|_{L^{2}}+\|f\|_{\dot{H}^{1}}+\left\Vert h_{i}\right\Vert _{Stri}+\left\Vert h_{i}\right\Vert _{Stri}^{5}\nonumber \\
	& +\left\Vert h_{i}\right\Vert _{D}\left\Vert h_{i}\right\Vert _{L_{x}^{\infty}L_{t}^{2}}+\left\Vert h_{i}^{S}\right\Vert _{D}\left\Vert h_{i}^{S}\right\Vert _{L_{x}^{\infty}L_{t}^{2}}\\
	& +\left\Vert F_{1}\right\Vert _{I}+\left\Vert F_{2}^{S}\right\Vert _{I}+\left\Vert F\right\Vert _{L_{t}^{1}L_{x}^{2}}.\nonumber 
\end{align}
\begin{align}
	\left\Vert h_{i+1}(x+\vec{v}t,t)\right\Vert _{D} & \lesssim\|g\|_{L^{2}}+\|f\|_{\dot{H}^{1}}+\left\Vert h_{i}\right\Vert _{Stri}+\left\Vert h_{i}\right\Vert _{Stri}^{5}\nonumber \\
	& +\left\Vert h_{i}\right\Vert _{D}\left\Vert h_{i}\right\Vert _{L_{x}^{\infty}L_{t}^{2}}+\left\Vert h_{i}^{S}\right\Vert _{D}\left\Vert h_{i}^{S}\right\Vert _{L_{x}^{\infty}L_{t}^{2}}\\
	& +\left\Vert F_{1}\right\Vert _{I}+\left\Vert F_{2}^{S}\right\Vert _{I}+\left\Vert F\right\Vert _{L_{t}^{1}L_{x}^{2}}.\nonumber 
\end{align}
Following the argument in Section 5 from \cite{GC3} and in the Appendix,
using the reversed Strichartz estimates to derive regular Strichartz
estimates, one has 
\begin{align}
	\left\Vert h_{i+1}\right\Vert _{Stri}+\sup_{t\geq t_{0}}\left\Vert h_{i+1}\right\Vert _{\dot{H}^{1}\times L^{2}}\lesssim & \|g\|_{L^{2}}+\|f\|_{\dot{H}^{1}}+\left\Vert h_{i}\right\Vert _{Stri}+\left\Vert h_{i}\right\Vert _{Stri}^{5}\nonumber \\
	& +\left\Vert h_{i}\right\Vert _{D}\left\Vert h_{i}\right\Vert _{L_{x}^{\infty}L_{t}^{2}}+\left\Vert h_{i}^{S}\right\Vert _{D}\left\Vert h_{i}^{S}\right\Vert _{L_{x}^{\infty}L_{t}^{2}}\nonumber \\
	& +\left\Vert F_{1}\right\Vert _{I}+\left\Vert F_{2}^{S}\right\Vert _{I}+\left\Vert F\right\Vert _{L_{t}^{1}L_{x}^{2}}.
\end{align}
By the computations in Remark \ref{rem:eg}, we can choose $t_{0}$
large enough, so that 
\begin{equation}
\left\Vert F\right\Vert _{L_{t}^{1}\left([t_{0},\infty),\,L_{x}^{2}\right)}+\left\Vert \left\langle x\right\rangle ^{\frac{1}{2}+\epsilon_{0}}\left(F_{2}^{S}+F_{1}\right)\right\Vert _{L_{t}^{2}\left([t_{0},\infty),\,L_{x}^{2}\right)}\ll\epsilon
\end{equation}
where $\epsilon$ is the small constant appearing in the contraction.

First we show that $h_{i+1}$ is bounded in all Strichartz norms and
the energy norm. 

Define the space $S$ as
\begin{equation}
S=\left\{ \left\Vert u\right\Vert _{Stri},\,\left\Vert u\right\Vert _{\dot{H}^{1}\times L^{2}},\,\left\Vert u\right\Vert _{L_{x}^{\infty}L_{t}^{2}},\,\left\Vert u^{S}\right\Vert _{L_{x}^{\infty}L_{t}^{2}},\,\left\Vert u\right\Vert _{D},\,\left\Vert u^{S}\right\Vert _{D}<\infty\right\} \label{eq:Sspace}
\end{equation}
By the iteration scheme:
\begin{align}
\partial_{tt}h_{i+1}-\Delta h_{i+1}+\left(V_{1}(x)+5W_{1}^{4}(x)\right)h_{i+1}+\left(V_{2}\left(x-\vec{v}t\right)+5W_{2}^{4}\left(x-\vec{v}t\right)\right)h_{i+1}\nonumber \\
=F_{1}(x,t)+F_{2}(x,t)+F(x,t)+N(h_{i},x,t)-a(x,t)h_{i}-h_{i}^{5}.
\end{align}with data
\begin{equation}
\left(h_{i}(t_{0}),\partial_{t}h_{i}(t_{0})\right)=\left(f,g\right)
\end{equation}
Then
\begin{align}
\left\Vert h_{i+1}-h_{j+1}\right\Vert _{S} & \lesssim\eta\left\Vert h_{i}-h_{j}\right\Vert _{L_{t}^{5}L_{x}^{10}}\label{eq:differ}\\
& +\left(\left\Vert h_{i}\right\Vert _{D}+\left\Vert h_{j}\right\Vert _{D}\right)\left\Vert h_{i}-h_{j}\right\Vert _{L_{x}^{\infty}L_{t}^{2}}\nonumber \\
& +\left(\left\Vert h_{i}^{S}\right\Vert _{D}+\left\Vert h_{j}^{S}\right\Vert _{D}\right)\left\Vert h_{i}^{S}-h_{j}^{S}\right\Vert _{L_{x}^{\infty}L_{t}^{2}}\nonumber \\
& +\left(\left\Vert h_{i}^{2}\right\Vert _{L_{t}^{\frac{3}{2}}L_{x}^{9}}+\left\Vert h_{j}^{2}\right\Vert _{L_{t}^{\frac{3}{2}}L_{x}^{9}}\right)\left\Vert h_{i}-h_{j}\right\Vert _{L_{t}^{3}L_{x}^{18}}\nonumber \\
& +\left(\left\Vert h_{i}^{3}\right\Vert _{L_{t}^{\frac{4}{3}}L_{x}^{4}}+\left\Vert h_{j}^{3}\right\Vert _{L_{t}^{\frac{4}{3}}L_{x}^{4}}\right)\left\Vert h_{i}-h_{j}\right\Vert _{L_{t}^{4}L_{x}^{12}}\nonumber \\
& +\left(\left\Vert h_{i}^{4}\right\Vert _{L_{t}^{\frac{5}{4}}L_{x}^{\frac{5}{2}}}+\left\Vert h_{j}^{4}\right\Vert _{L_{t}^{\frac{5}{4}}L_{x}^{\frac{5}{2}}}\right)\left\Vert h_{i}-h_{j}\right\Vert _{L_{t}^{5}L_{x}^{10}}\nonumber 
\end{align}
Let $h_{-1}=0$. We have
\begin{align}
\partial_{tt}h_{0}-\Delta h_{0}+\left(V_{1}(x)+5W_{1}^{4}(x)\right)h_{0}+\left(V_{2}\left(x-\vec{v}t\right)+5W_{2}^{4}\left(x-\vec{v}t\right)\right)h_{0}\nonumber \\
=F_{1}(x,t)+F_{2}(x,t)+F(x,t).
\end{align} 
with 
\begin{equation}
\left(h_{0}(t_{0}),\partial_{t}h_{0}(t_{0})\right)=\left(f,g\right)
\end{equation}
such that 
\begin{equation}
\left\Vert \left(f,g\right)\right\Vert _{\dot{H}^{1}\times L^{2}}\ll\epsilon.
\end{equation}
Then by our Strichartz estimates from Theorem \ref{thm:StriCharB}, Theorem \ref{thm:EndRStriCB} and Theorem \ref{thm:EnergyCharge}, one has
\begin{equation}
\left\Vert h_{0}\right\Vert _{S}\leq\epsilon\ll1.
\end{equation}
By induction, suppose that
\begin{equation}
\left\Vert h_{j}\right\Vert _{S}\leq2\epsilon\ll1.
\end{equation}
By similar computations to the above, we can conclude that 
\begin{align}
	\left\Vert h_{j+1}-h_{0}\right\Vert _{S} & \lesssim\eta\left\Vert h_{j}\right\Vert _{L_{t}^{5}L_{x}^{10}}\label{eq:bound1}\\
	& +\left(\left\Vert h_{j}\right\Vert _{D}\left\Vert h_{j}\right\Vert _{L_{x}^{\infty}L_{t}^{2}}+\left\Vert h_{j}^{S}\right\Vert _{D}\left\Vert h_{j}^{S}\right\Vert _{L_{x}^{\infty}L_{t}^{2}}\right)\nonumber \\
	& +\left(\left\Vert h_{j}^{2}\right\Vert _{L_{t}^{\frac{3}{2}}L_{x}^{9}}\left\Vert h_{j}\right\Vert _{L_{t}^{3}L_{x}^{18}}\right)\nonumber \\
	& +\left(\left\Vert h_{j}^{3}\right\Vert _{L_{t}^{\frac{4}{3}}L_{x}^{4}}\left\Vert h_{j}\right\Vert _{L_{t}^{4}L_{x}^{12}}\right)\nonumber \\
	& +\left(\left\Vert h_{j}^{4}\right\Vert _{L_{t}^{\frac{5}{4}}L_{x}^{\frac{5}{2}}}\left\Vert h_{j}\right\Vert _{L_{t}^{5}L_{x}^{10}}\right)\nonumber \\
	\lesssim & \eta\left\Vert h_{j}\right\Vert _{S}+\left\Vert h_{j}\right\Vert _{S}^{2}+\left\Vert h_{j}\right\Vert _{S}^{3}+\left\Vert h_{j}\right\Vert _{S}^{4}+\left\Vert h_{j}\right\Vert _{S}^{5}\nonumber \\
	\leq & C\left(2\eta\epsilon+4\epsilon^{2}+8\epsilon^{3}+16\epsilon^{4}+32\epsilon^{5}\right)\nonumber \\
	\leq & C\left(2\eta+4\epsilon+8\epsilon^{2}+16\epsilon^{3}+32\epsilon^{4}\right)\epsilon.\nonumber 
\end{align}
One just needs to pick $\epsilon$ small and $\eta$ small such that
\begin{equation}
C\left(\eta+4\epsilon+8\epsilon^{2}+16\epsilon^{3}+32\epsilon^{4}\right)<\frac{1}{2}.
\end{equation}
\begin{equation}
\eta=\left\Vert a(x,t)\right\Vert _{L_{t}^{\frac{5}{4}}\left([t_{0},\infty),\,L_{x}^{\frac{5}{2}}\right)}.
\end{equation}
Note that 
\begin{equation}
\left\Vert 20W_{1}^{3}(x)W_{2}\left(x-\vec{v}t\right)+30W_{1}^{2}(x)W_{2}^{2}\left(x-\vec{v}t\right)+20W_{1}(x)W_{2}^{3}\left(x-\vec{v}t\right)\right\Vert _{L_{t}^{\frac{5}{4}}\left([t_{0},\infty),\,L_{x}^{\frac{5}{2}}\right)}
\end{equation}
can be made sufficiently small provided $t_{0}$ is large enough.
Therefore
\begin{equation}
\left\Vert a(x,t)\right\Vert _{L_{t}^{\frac{5}{4}}\left([t_{0},\infty),\,L_{x}^{\frac{5}{2}}\right)}
\end{equation}
can be made arbitrarily small provided $t_{0}$ is large.

Therefore, by induction, we have  
\begin{equation}
\left\Vert h_{j+1}\right\Vert _{S}\leq2\epsilon.
\end{equation}
Next we show that the above construction gives a contraction. By almost
the same computations as above,
\begin{align}
	\left\Vert h_{i+1}-h_{i}\right\Vert _{S} & \lesssim\eta\left\Vert h_{i}-h_{i-1}\right\Vert _{L_{t}^{5}L_{x}^{10}}\label{eq:bound2}\\
	& +\left(\left\Vert h_{i}\right\Vert _{D}+\left\Vert h_{i-1}\right\Vert _{D}\right)\left\Vert h_{i}-h_{i-1}\right\Vert _{L_{x}^{\infty}L_{t}^{2}}\nonumber \\
	& +\left(\left\Vert h_{i}^{S}\right\Vert _{D}+\left\Vert h_{i-1}^{S}\right\Vert _{D}\right)\left\Vert h_{i}^{S}-h_{i-1}^{S}\right\Vert _{L_{x}^{\infty}L_{t}^{2}}\nonumber \\
	& +\left(\left\Vert h_{i}^{2}\right\Vert _{L_{t}^{\frac{3}{2}}L_{x}^{9}}+\left\Vert h_{i-1}^{2}\right\Vert _{L_{t}^{\frac{3}{2}}L_{x}^{9}}\right)\left\Vert h_{i}-h_{i-1}\right\Vert _{L_{t}^{3}L_{x}^{18}}\nonumber \\
	& +\left(\left\Vert h_{i}^{3}\right\Vert _{L_{t}^{\frac{4}{3}}L_{x}^{4}}+\left\Vert h_{i-1}^{3}\right\Vert _{L_{t}^{\frac{4}{3}}L_{x}^{4}}\right)\left\Vert h_{i}-h_{i-1}\right\Vert _{L_{t}^{4}L_{x}^{12}}\nonumber \\
	& +\left(\left\Vert h_{i}^{4}\right\Vert _{L_{t}^{\frac{5}{4}}L_{x}^{\frac{5}{2}}}+\left\Vert h_{i-1}^{4}\right\Vert _{L_{t}^{\frac{5}{4}}L_{x}^{\frac{5}{2}}}\right)\left\Vert h_{i}-h_{i-1}\right\Vert _{L_{t}^{5}L_{x}^{10}}\nonumber \\
	\leq & \frac{1}{2}\left\Vert h_{i}-h_{i-1}\right\Vert _{S}.\nonumber 
\end{align}
Therefore by the Banach fixed-point theorem, there exists a unique
solution to 
\begin{align}
	\partial_{tt}h-\Delta h+h^{5}+\left(V_{1}(x)+5W_{1}^{4}(x)\right)h+\left(V_{2}\left(x-\vec{v}t\right)+5W_{2}^{4}\left(x-\vec{v}t\right)\right)h+a(x,t)h\nonumber \\
	=F_{1}(x,t)+F_{2}(x,t)+F(x,t)+N(h,x,t)
\end{align}
such that 
\begin{equation}
\left\Vert h\right\Vert _{S}\leq2\epsilon.
\end{equation}
Hence by Theorem \ref{thm:scattering},
\begin{equation}
u(t)-W_{1}\left(x\right)-W_{2}\left(x-\vec{v}t\right)
\end{equation}
scatters to free wave.
\end{proof}
\begin{rem}
	The quadratic term can also be handled by estimate \eqref{eq:EndStrichartz}
	from Theorem \ref{thm:StriCharB}. 
\end{rem}
Secondly, we show the existence of the purely multi-soliton solution.
We solve the equation for $h$ backwards from infinity.
\begin{proof}[Proof of Theorem \ref{thm:existence}]
Again, we consider 
	\begin{equation}
	h(t):=u(t)-W_{1}\left(x\right)-W_{2}\left(x-\vec{v}t\right),
	\end{equation}
	then
	\begin{align}
		\partial_{tt}h-\Delta h+h^{5}+\left(V_{1}(x)+5W_{1}^{4}(x)\right)h+\left(V_{2}\left(x-\vec{v}t\right)+5W_{2}^{4}\left(x-\vec{v}t\right)\right)h\nonumber \\
		=:F_{1}(x,t)+F_{2}(x,t)+F(x,t)+N(h,x,t)-a(x,t)h\\
		=:\mathrm{F}(h,x,t)
	\end{align}
	and 
	\begin{equation}
	\partial_{tt}h-\Delta h=:G(h,x,t),\label{eq:scattE}
	\end{equation}
	\[
	\left\Vert h(t)\right\Vert _{\dot{H}^{1}\times L^{2}}\rightarrow0,\,t\rightarrow\infty.
	\]
	As the beginning of this section, we set $A=\sqrt{-\Delta}$ and
	notice that $h$ solves \eqref{eq:scattE} if and only if 
	\begin{equation}
	H:=Ah+i\partial_{t}h
	\end{equation}
	satisfies 
	\begin{equation}
	i\partial_{t}H=AH+G(h,x,t),
	\end{equation}
	\begin{equation}
	H(t)\rightarrow0,\,\,t\rightarrow\infty
	\end{equation}
	in the sense of $L^{2}$ norm.
	
	By Duhamel's formula, for fixed $T$ 
	\begin{equation}
	H(t)=e^{i\left(t-T\right)A}H(T)-i\int_{T}^{t}e^{-i\left(t-s\right)A}G(h,\cdot,s)\,ds.
	\end{equation}
	Letting $T$ go to $\infty$, we know $H(T)\rightarrow0$, so 
	\begin{equation}
	H(t):=i\int_{t}^{\infty}e^{i\left(t-s\right)A}G\left(h,\cdot,s\right)\,ds.
	\end{equation}
	By construction, we just need to show $H(t)$ is well-defined in $L^{2}$,
	then automatically, 
	\begin{equation}
	\left\Vert u(t)-W_{1}\left(x\right)-W_{2}\left(x-\vec{v}t\right)\right\Vert _{\dot{H}^{1}\times L^{2}}\rightarrow0.
	\end{equation}
	It suffices to show 
	\begin{equation}
	H(t)=i\int_{t}^{\infty}e^{i\left(t-s\right)A}G\left(h,\cdot,s\right)\,ds\in L^{2},\label{eq:backeq}
	\end{equation}
	is well-defined. We show the existence of such a solution for $t\geq t_{0}$
	provided $t_{0}$ is large enough. This can be done by a similar contraction
	argument to the previous proof.
	
	Indeed, we consider 
	\begin{align}
		\partial_{tt}h_{i+1}-\Delta h_{i+1}+\left(V_{1}(x)+5W_{1}^{4}(x)\right)h_{i+1}+\left(V_{2}\left(x-\vec{v}t\right)+5W_{2}^{4}\left(x-\vec{v}t\right)\right)h_{i+1}\nonumber \\
		=F_{1}(x,t)+F_{2}(x,t)+F(x,t)+N(h_{i},x,t)-a(x,t)h_{i}-h_{i}^{5}.
	\end{align}
	Again setting $h_{-1}=0$, then by Duhamel's formula and the equation
	\eqref{eq:backeq}, we have 
	\begin{align}
		h_{0}(x,t) & =\int_{t}^{\infty}U(t,s)\left(F_{1}(\cdot,t)+F_{2}(\cdot,s)+F(\cdot,s)\right)\,ds\nonumber \\
		& =\int_{t}^{\infty}U(t,s)\mathrm{F}(h_{-1},\cdot,s)\,ds.
	\end{align}
	Then by estimates \eqref{eq:Strichartz} and \eqref{eq:RStrichartzS}
	from Theorem \ref{thm:EndRStriCB}, one has 
	\begin{align}
		\sup_{x\in\mathbb{R}^{3}}\left(\int_{t_{1}}^{\infty}\left|h_{0}(x,t)\right|^{2}dt\right)^{\frac{1}{2}} & \lesssim\left\Vert F\right\Vert _{L_{t}^{1}[t_{1},\infty)L_{x}^{2}}+\left\Vert F_{1}\right\Vert _{I_{t_{1}}}+\left\Vert F_{2}^{S}\right\Vert _{I_{t_{1}}}\nonumber \\
		& \lesssim\frac{1}{\sqrt{t_{1}}}
	\end{align}
	and 
	\begin{align}
		\sup_{x\in\mathbb{R}^{3}}\left(\int_{t_{1}}^{\infty}\left|h_{0}(x+\vec{v}t,t)\right|^{2}dt\right)^{\frac{1}{2}} & \lesssim\left\Vert F\right\Vert _{L_{t}^{1}[t_{1},\infty)L_{x}^{2}}+\left\Vert F_{1}\right\Vert _{I_{t_{1}}}+\left\Vert F_{2}^{S}\right\Vert _{I_{t_{1}}}\nonumber \\
		& \lesssim\frac{1}{\sqrt{t_{1}}},
	\end{align}
	where $I_{t_{1}}$ is the space obtained by restricting space $I$
	given by \eqref{eq:Ispace-1} onto $[t_{1},\infty)$.
	
	Then by the argument in Lemma \ref{lem:energyL} and Remark \ref{rem:eg},
	we write 
	\begin{align}
		\partial_{tt}h_{0}-\Delta h_{0} & =-\left(V_{1}(x)+5W_{1}^{4}(x)\right)h_{0}+\left(V_{2}\left(x-\vec{v}t\right)+5W_{2}^{4}\left(x-\vec{v}t\right)\right)h_{0}\nonumber \\
		& +F_{1}(x,t)+F_{2}(x,t)+F(x,t)\nonumber \\
		& =:\mathrm{D}\left(h_{-1},x,t\right)
	\end{align}
and then conclude that 
	\begin{equation}
	\left\Vert h_{0}(t_{1})\right\Vert _{\dot{H}^{1}\times L^{2}}\lesssim\frac{1}{\sqrt{t_{1}}}.
	\end{equation}
	It follows that 
	\begin{equation}
	\left\Vert h_{0}\right\Vert _{S_{t_{1}}}\lesssim\frac{1}{\sqrt{t_{1}}}
	\end{equation}
	where $S_{t_{1}}$ is the space given by \eqref{eq:Sspace} restricted
	onto onto $[t_{1},\infty)$.
	
	Next, we can run the contraction argument as the proof Theorem \ref{thm:Stability}.
	We consider the iteration give by the following formula.	
	\begin{equation}
	h_{i+1}(x,t)=\int_{t}^{\infty}U(t,s)\mathrm{F}(h_{i},\cdot,s)\,ds.
	\end{equation}
	Then by the same computations as \eqref{eq:differ} restricted onto
	$[t_{1},\infty)$, one has
	\begin{align}
		\left\Vert h_{i+1}-h_{j+1}\right\Vert _{S_{t_{1}}} & \lesssim\eta\left\Vert h_{i}-h_{j}\right\Vert _{L_{t}^{5}[t_{1},\infty)L_{x}^{10}}\nonumber\\
		& +\left(\left\Vert h_{i}\right\Vert _{D_{t_{1}}}+\left\Vert h_{j}\right\Vert _{D_{t_{1}}}\right)\left\Vert h_{i}-h_{j}\right\Vert _{L_{x}^{\infty}L_{t}^{2}[t_{1},\infty)}\label{eq:differ-1} \\
		& +\left(\left\Vert h_{i}^{S}\right\Vert _{D_{t_{1}}}+\left\Vert h_{j}^{S}\right\Vert _{D_{t_{1}}}\right)\left\Vert h_{i}^{S}-h_{j}^{S}\right\Vert _{L_{x}^{\infty}L_{t}^{2}[t_{1},\infty)}\nonumber \\
		& +\left(\left\Vert h_{i}^{2}\right\Vert _{L_{t}^{\frac{3}{2}}[t_{1},\infty)L_{x}^{9}}+\left\Vert h_{j}^{2}\right\Vert _{L_{t}^{\frac{3}{2}}[t_{1},\infty)L_{x}^{9}}\right)\left\Vert h_{i}-h_{j}\right\Vert _{L_{t}^{3}[t_{1},\infty)L_{x}^{18}}\nonumber \\
		& +\left(\left\Vert h_{i}^{3}\right\Vert _{L_{t}^{\frac{4}{3}}[t_{1},\infty)L_{x}^{4}}+\left\Vert h_{j}^{3}\right\Vert _{L_{t}^{\frac{4}{3}}[t_{1},\infty)L_{x}^{4}}\right)\left\Vert h_{i}-h_{j}\right\Vert _{L_{t}^{4}[t_{1},\infty)L_{x}^{12}}\nonumber \\
		& +\left(\left\Vert h_{i}^{4}\right\Vert _{L_{t}^{\frac{5}{4}}[t_{1},\infty)L_{x}^{\frac{5}{2}}}+\left\Vert h_{j}^{4}\right\Vert _{L_{t}^{\frac{5}{4}}[t_{1},\infty)L_{x}^{\frac{5}{2}}}\right)\left\Vert h_{i}-h_{j}\right\Vert _{L_{t}^{5}[t_{1},\infty)L_{x}^{10}}.\nonumber 
	\end{align}
	For all $t_{1}$ such that 
	\begin{equation}
	t_{0}\ll t_{1},\,\,\frac{1}{\sqrt{t_{1}}}\ll\epsilon.
	\end{equation}
	where $t_{0}$ and $\epsilon$ are constants depend on prescribed
	constants as in the proof of Theorem \ref{thm:Stability}. 
	
	We can conclude that 
	\begin{equation}
	\left\Vert h_{i}-h_{0}\right\Vert _{S_{t_{1}}}\lesssim\frac{1}{\sqrt{t_{1}}}
	\end{equation}
	\begin{align}
		\left\Vert h_{i+1}-h_{i}\right\Vert _{S_{t_{1}}} & \leq\frac{1}{2}\left\Vert h_{i}-h_{i-1}\right\Vert _{S_{t_{1}}}
	\end{align}
	Therefore by the Banach fixed-point theorem, there exists a unique
	solution to 
	\begin{equation}
	h(x,t)=\int_{t}^{\infty}U(t,s)\mathrm{F}(h,\cdot,s)\,ds
	\end{equation}
	such that 
	\begin{equation}
	\left\Vert h(t)\right\Vert _{\dot{H}^{1}\times L^{2}}\lesssim\frac{1}{\sqrt{t}}.
	\end{equation}
	Therefore, we conclude that if we write 
	\begin{equation}
	U[t]=\left(u,u_{t}\right)^{t},\,W[t]=\left(W_{1}\left(x\right)-W_{2}\left(x-\vec{v}t\right),\,\partial_{t}\left(W_{1}\left(x\right)-W_{2}\left(x-\vec{v}t\right)\right)\right)^{t},
	\end{equation}
	there exists a solution $u$ to 
	\begin{equation}
	\partial_{tt}u-\Delta u+V_{1}\left(x\right)u+V_{2}\left(x-\vec{v}t\right)u+u^{5}=0,
	\end{equation}
	such that 
	\begin{equation}
	\lim_{t\rightarrow\infty}\left\Vert U[t]-W[t]\right\Vert _{\dot{H}^{1}\times L^{2}}= 0.
	\end{equation}
	Moreover, we have the decay rate 
	\begin{equation}
	\left\Vert U[t]-W[t]\right\Vert _{\dot{H}^{1}\times L^{2}}\lesssim\frac{1}{\sqrt{t}}.
	\end{equation}
	as $t\rightarrow\infty$. We are done.
\end{proof}

\subsection{Unstable solitons}

To finish this section, we discuss the case that we have some unstable
solitons. From the discussion above, the linear model still plays a
pivotal role. But in this case, the analysis is much more involved
due to the unstable structure. Consider 
\begin{equation}
\partial_{tt}u-\Delta u+V_{1}(x)u+V_{2}\left(x-\vec{v}t\right)u+u^{5}=0
\end{equation}
and
\begin{equation}
h=u(t)-Q_{1}\left(x\right)-Q_{2}\left(x-\vec{v}t\right).
\end{equation}
where both $Q_{1}$ and $Q_{2}$ are unstable. For simplicity, suppose
that 
\begin{equation}
L_{Q_{1}}=-\Delta+V_{1}+5Q_{1}^{4}
\end{equation}
has one negative eigenvalue and zero is neither an eigenvalue nor
resonance. Also suppose 
\begin{equation}
L_{Q_{2}}=-\Delta+V_{2}^{v}+5\left(Q_{2}^{v}\right)^{4}
\end{equation}
has one negative eigenvalue and zero is neither an eigenvalue nor
resonance.

With $\lambda>0,\,\,\mu>0,$
\begin{equation}
L_{Q_{1}}w=-\lambda^{2}w,\,\,L_{Q_{2}}m=-\mu^{2}m.
\end{equation}
$w$ and $m$ decay exponentially by Agmon's estimate, see \cite{Agmon,GC2}.
The analysis can be easily adapted to the most general situation. 

Set 
\begin{equation}
Q[t]=\left(Q_{1}\left(x\right)-Q_{2}\left(x-\vec{v}t\right),\,\partial_{t}\left(Q_{1}\left(x\right)-Q_{2}\left(x-\vec{v}t\right)\right)\right)^{t}.
\end{equation}

\begin{thm}
	\label{thm:condstab}Suppose that $0<\epsilon\ll1$ is small enough
	and $1\ll t_{0}$ is large enough. There is a codimension $1+1$ smooth manifold
	$\mathbb{\mathcal{M}}\in\dot{H}\left(\mathbb{R}^{3}\right)\times L^{2}\left(\mathbb{R}^{3}\right)$
	around the $\epsilon$ neighborhood around $Q[t_{0}]$ such that if
	$u\in\mathcal{M}$ solve
	\begin{equation}
	\partial_{tt}u-\Delta u+V_{1}(x)u+V_{2}\left(x-\vec{v}t\right)u+u^{5}=0,
	\end{equation}
	and 
	\begin{equation}
	\left\Vert U[t_{0}]-Q[t_{0}]\right\Vert _{\dot{H}^{1}\times L^{2}}\leq\epsilon.
	\end{equation}
	Then there exists free data
	\begin{equation}
	U_{0}[0]=\left(f_{0},g_{0}\right)^{t}\in\dot{H}^{1}\times L^{2}
	\end{equation}
	such that
	\begin{equation}
	\left\Vert U[t]-Q[t]-e^{tJH_{F}}U_{0}[0]\right\Vert _{\dot{H}^{1}\times L^{2}}\rightarrow0,\,\,t\rightarrow\infty.
	\end{equation}
	In other words, the error $u(t)-Q_{1}\left(x\right)-Q_{2}\left(x-\vec{v}t\right)$\textup{
	}scatters to the free wave.
\end{thm}
\begin{proof}
	 As in the stable case, by construction,
	we have 
	\begin{align}
		\partial_{tt}h-\Delta h+h^{5}+\left(V_{1}(x)+5Q_{1}^{4}(x)\right)h+\left(V_{2}\left(x-\vec{v}t\right)+5Q_{2}^{4}\left(x-\vec{v}t\right)\right)h+a(x,t)h\nonumber \\
		=F_{1}(x,t)+F_{2}(x,t)+F(x,t)+N(h,x,t)
	\end{align}
	Again, we consider the evolution starting from $t_{0}$ where $t_{0}$
	is large enough and only depends on prescribed constants. When dealing
	with unstable solitons, we need to make sure the evolution under the
	iteration is a scattering in each iterated step. So we need to modify
	the data after each iteration. 
	
	As in the stable case, we consider the iteration:
	\begin{align}
		\partial_{tt}h_{i+1}-\Delta h_{i+1}+\left(V_{1}(x)+5Q_{1}^{4}(x)\right)h_{i+1}+\left(V_{2}\left(x-\vec{v}t\right)+5Q_{2}^{4}\left(x-\vec{v}t\right)\right)h_{i+1}\nonumber \\
		=F_{1}(x,t)+F_{2}(x,t)+F(x,t)+N(h_{i},x,t)-a(x,t)h_{i}-h_{i}^{5}\label{eq:iterate}
	\end{align}
	and 
	\begin{equation}
	h_{-1}\equiv0.
	\end{equation}
	Denote 
	\begin{equation}
	\mathbf{\mathrm{V}_{i}}=V_{i}(x)+5Q_{i}^{4}(x),
	\end{equation}
	\begin{equation}
	\mathrm{D}(h_{i},x,t)=N(h_{i},x,t)-a(x,t)h_{i}-h_{i}^{5}
	\end{equation}
	and 
	\begin{equation}
	\mathrm{F}(h_{i},x,t)=F_{1}(x,t)+F_{2}(x,t)+F(x,t)+\mathrm{D}(h_{i},x,t).
	\end{equation}
	Decompose $h_{i}$ into three pieces:
	\begin{equation}
	h_{i}(x,t)=a_{i}(t)w(x)+b_{i}\left(\gamma(t-vx_{1})\right)m_{v}\left(x,t\right)+r_{i}(x,t)\label{eq:evolution}
	\end{equation}
	where
	\begin{equation}
	m_{v}(x,t)=m\left(\gamma\left(x_{1}-\vec{v}t\right),x_{2},x_{3}\right).
	\end{equation}
	We notice that 
	\begin{equation}
	P_{c}\left(H_{1}\right)r_{i}=r_{i}
	\end{equation}
	and 
	\begin{equation}
	P_{c}\left(H_{2}\right)\left(r_{i}\right)_{L}=\left(r_{i}\right)_{L}
	\end{equation}
	where the Lorentz transformation $L$ makes $\mathbf{\mathrm{V}}_{2}$ stationary. 
	
	Under the iteration, for the initial data, we impose that for $i\geq1$.
	\begin{equation}
	a(t_{0})=a_{i}(t_{0})=a_{0}(t_{0}),
	\end{equation}
	\begin{equation}
	b\left(\sqrt{1-\left|v\right|^{2}}t_{0}\right)=b_{i}\left(\sqrt{1-\left|v\right|^{2}}t_{0}\right)=b_{0}\left(\sqrt{1-\left|v\right|^{2}}t_{0}\right)
	\end{equation}
	and 
	\begin{equation}
	r(x,t_{0})=r_{i}\left(x,t_{0}\right)=r_{0}\left(x,t_{0}\right).
	\end{equation}
	We first analyze the behavior of the bound states as in \cite{GC3}.
	Plugging the evolution \eqref{eq:evolution} into the equation \eqref{eq:iterate}
	and taking inner product with $w$, we get 
	\[
	\ddot{a}_{i}(t)-\lambda^{2}a_{i}(t)+\left\langle \mathrm{V}_{2}\left(x-\vec{v}t\right)h_{i},w\right\rangle =\left\langle \mathrm{F}(h_{i},x,t),w\right\rangle 
	\]
	Denote 
	\begin{equation}
	N_{i}(t):=\left\langle \mathrm{F}(h_{i},x,t),w\right\rangle -\left\langle \mathrm{V}_{2}\left(x-\vec{v}t\right)h_{i},w\right\rangle .
	\end{equation}
	Then
	\begin{equation}
	a_{i}(t)=\frac{e^{\lambda t}}{2}\left[a_{i}(0)+\frac{1}{\lambda}\dot{a}_{i}(0)+\frac{1}{\lambda}\int_{t_{0}}^{t}e^{-\lambda s}N_{i}(s)\,ds\right]+R(t)
	\end{equation}
	where 
	\begin{equation}
	\left|R(t)\right|\lesssim e^{-\beta t},
	\end{equation}
	for some positive constant $\beta>0$. Therefore, the stability condition
	from scattering conditions in the sense of Definition \ref{AO} forces
	\begin{equation}
	a_{i}(t_{0})+\frac{1}{\lambda}\dot{a}_{i}(t_{0})+\frac{1}{\lambda}\int_{t_{0}}^{\infty}e^{-\lambda s}N_{i}(s)\,ds=0.\label{eq:stability}
	\end{equation}
	So as the discussion in \cite{GC3}, given $a(t_{0})$, there is a
	unique $\dot{a}(t_{0})$ such that the stability condition \eqref{eq:stability}
	is satisfied. Similar results hold for $b_{0}\left(t\right)$ up to
	performing a Lorentz transformation. These stability conditions will
	ensure that $h_{i}$ is a scattering state. Therefore, we can employ
	the estimates from Theorem \ref{thm:StriCharB}, Theorem \ref{thm:EnergyCharge}
	and Theorem \ref{thm:EndRStriCB} as in the proof of Theorem \ref{thm:Stability}.
	
	Consider the iteration for $\dot{a}_{i}(t_{0})$,
	\begin{align}
		\dot{a}_{i+1}(t_{0})-\dot{a}_{j+1}(t_{0}) & =-\int_{t_{0}}^{\infty}e^{-\lambda s}\left(N_{i+1}(s)-N_{j+1}(s)\right)\,ds.
	\end{align}
	Note that 
	\begin{align}
		\left|N_{i+1}(s)-N_{j+1}(s)\right| & \lesssim\left|\left\langle \mathrm{V}_{2}\left(x-\vec{v}t\right)\left(h_{i+1}-h_{j+1}\right),w\right\rangle \right|\nonumber \\
		& +\left|\left\langle \mathrm{D}(h_{i},x,t)-\mathrm{D}(h_{j},x,t),w\right\rangle \right|.
	\end{align}
	Then for $1\leq p\leq2$, by Minkowski's inequality and H\"older's
	inequality, 
	\begin{align}
		\left\Vert \left|\left\langle \mathrm{V}_{2}\left(x-\vec{v}t\right)\left(h_{i+1}-h_{j+1}\right),w\right\rangle \right|\right\Vert _{L_{t}^{p}[t_{0},\infty)}\qquad\nonumber \\
		\lesssim\left|\left\langle \left\Vert \left|\mathrm{V}_{2}\left(x-\vec{v}t\right)\left(h_{i+1}-h_{j+1}\right)\right|\right\Vert _{L_{t}^{p}[t_{0},\infty)},w\right\rangle \right|\nonumber \\
		\lesssim\frac{1}{\left\langle t_{0}\right\rangle }\left\Vert h_{i+1}-h_{j+1}\right\Vert _{L_{x}^{\infty}L_{t}^{2}[t_{0},\infty)}.\label{eq:differ1}
	\end{align}
	To estimate the difference between $\mathrm{D}(h_{i},x,t)$ and $\mathrm{D}(h_{j},x,t)$,
	we do the same computations as in the stable solitons case, see \eqref{eq:differ},
	\begin{align}
		\left|\left\langle \mathrm{D}(h_{i},x,t)-\mathrm{D}(h_{j},x,t),w\right\rangle \right|_{L_{t}^{p}[t_{0},\infty)} & \lesssim\left\Vert h_{i}-h_{j}\right\Vert _{S}.\label{eq:differ2}
	\end{align}
	Then combine \eqref{eq:differ1} and \eqref{eq:differ2} together, one
	has
	\begin{align}
		\left|\dot{a}_{i+1}(t_{0})-\dot{a}_{j+1}(t_{0})\right| & \lesssim e^{-\lambda t_{0}}\left(\int_{t_{0}}^{\infty}\left|N_{i+1}(s)-N_{j+1}(s)\right|^{2}ds\right)^{\frac{1}{2}}\nonumber \\
		& \lesssim e^{-\lambda t_{0}}\left(\frac{1}{\left\langle t_{0}\right\rangle }\left\Vert h_{i+1}-h_{j+1}\right\Vert _{S}+\left\Vert h_{i}-h_{j}\right\Vert _{S}\right)
	\end{align}
	Similarly, 
	\begin{align}
		\left|\dot{b}_{i+1}\left(\sqrt{1-\left|v\right|^{2}}t_{0}\right)-\dot{b}_{j+1}\left(\sqrt{1-\left|v\right|^{2}}t_{0}\right)\right|\qquad\qquad\nonumber \\
		\lesssim e^{-\sqrt{1-\left|v\right|^{2}}\mu t_{0}}\left(\frac{1}{\left\langle t_{0}\right\rangle }\left\Vert h_{i+1}-h_{j+1}\right\Vert _{S}+\left\Vert h_{i}-h_{j}\right\Vert _{S}\right).
	\end{align}
	with 
	\begin{equation}
	\left\Vert h_{0}(t_{0})\right\Vert _{\dot{H}^{1}\times L^{2}}\ll\epsilon
	\end{equation}
	Then by our Strichartz estimates 
	\begin{equation}
	\left\Vert h_{0}\right\Vert _{S}\lesssim\epsilon\ll1.
	\end{equation}
	Next we consider the estimate in our iteration scheme as \eqref{eq:differ},
	\eqref{eq:bound1} and \eqref{eq:bound2}. It suffices to estimate:
	\begin{align}
		\left\Vert h_{i+1}-h_{j+1}\right\Vert _{S} & \leq\left|\dot{a}_{i+1}(t_{0})-\dot{a}_{j+1}(t_{0})\right|\nonumber \\
		& +\left|\dot{b}_{j+1}\left(\sqrt{1-\left|v\right|^{2}}t_{0}\right)-\dot{b}_{i+1}\left(\sqrt{1-\left|v\right|^{2}}t_{0}\right)\right|\nonumber \\
		& +\frac{1}{4}\left\Vert h_{i}-h_{j}\right\Vert _{S}
	\end{align}
	Therefore as the stable case, \eqref{eq:bound1} and \eqref{eq:bound2},
	we haves
	\begin{equation}
	\left\Vert h_{i+1}-h_{i}\right\Vert _{S}\leq\frac{1}{2}\left\Vert h_{i}-h_{i-1}\right\Vert _{S}
	\end{equation}
	and
	\begin{equation}
	\left|\dot{a}_{i+1}(t_{0})-\dot{a}_{j+1}(t_{0})\right|+\left|\dot{b}_{i+1}\left(\sqrt{1-\left|v\right|^{2}}t_{0}\right)-\dot{b}_{j+1}\left(\sqrt{1-\left|v\right|^{2}}t_{0}\right)\right|\leq\frac{1}{8}\left\Vert h_{i}-h_{j}\right\Vert _{S}
	\end{equation}
	Therefore by the Banach fixed-point theorem, there exist $h$, $\dot{a}(t_{0})$
	and $\dot{b}\left(\sqrt{1-\left|v\right|^{2}}t_{0}\right)$ such that
	\begin{equation}
	\left\Vert h_{i}-h\right\Vert _{S}\rightarrow0,
	\end{equation}
	\begin{equation}
	\,\left|\dot{a}_{i}(t_{0})-\dot{a}(t_{0})\right|\rightarrow0,
\end{equation}
and
\begin{equation}
	\left|\dot{b}_{i}\left(\sqrt{1-\left|v\right|^{2}}t_{0}\right)-\dot{b}\left(\sqrt{1-\left|v\right|^{2}}t_{0}\right)\right|\rightarrow0
	\end{equation}
as $i\rightarrow\infty$.

	Moreover,
	\begin{equation}
	a(t_{0})+\frac{1}{\lambda}\dot{a}(t_{0})+\frac{1}{\lambda}\int_{t_{0}}^{\infty}e^{-\lambda s}N(s)\,ds=0
	\end{equation}
	where 
	\begin{equation}
	N(t):=\left\langle \mathrm{F}(h,x,t),w\right\rangle -\left\langle \mathrm{V}_{2}\left(x-\vec{v}t\right)h,w\right\rangle ,
	\end{equation}
	the same condition holds for $b(t)$. 
	
	It follows that $h$ is scattering state and satisfies
	\begin{align}
		\partial_{tt}h-\Delta h+\left(V_{1}(x)+5Q_{1}^{4}(x)\right)h+\left(V_{2}\left(x-\vec{v}t\right)+5Q_{2}^{4}\left(x-\vec{v}t\right)\right)h\nonumber \\
		=F_{1}(x,t)+F_{2}(x,t)+F(x,t)+N(h,x,t)-a(x,t)h-h^{5},
	\end{align}
	and 
	\begin{equation}
	\left\Vert h\right\Vert _{S}\lesssim\epsilon.
	\end{equation}
   Hence 
	\begin{equation}
	u(t)-Q_{1}\left(x\right)-Q_{2}\left(x-\vec{v}t\right)
	\end{equation}
	scatters to free wave.
	
		Notice that the above construction depends on the data smoothly.
\end{proof}
\begin{rem}
	\label{rem:general1}We can also  consider the most general case. Suppose that 
	\begin{equation}
	L_{Q_{1}}=-\Delta+V_{1}+5Q_{1}^{4}
	\end{equation}
	has $k_{1}$ negative eigenvalues and zero is neither an eigenvalue
	nor resonance. Also suppose 
	\begin{equation}
	L_{Q_{2}}=-\Delta+V_{2}^{v}+5\left(Q_{2}^{v}\right)^{4}
	\end{equation}
	has $k_{2}$ negative eigenvalues and zero is neither an eigenvalue
	nor resonance.
	
	Then by similar arguments as above, we can obtain the general conditional
	stability results: Suppose that $0<\epsilon\ll1$ is small enough
	and $1\ll t_{0}$ is large enough. There is a codimension $k_{1}+k_{2}$
	smooth manifold $\mathbb{\mathcal{M}}\in\dot{H}\left(\mathbb{R}^{3}\right)\times L^{2}\left(\mathbb{R}^{3}\right)$
	around the $\epsilon$ neighborhood around $Q[t_{0}]$ such that if
	$u\in\mathcal{M}$ solve
	\begin{equation}
	\partial_{tt}u-\Delta u+V_{1}(x)u+V_{2}\left(x-\vec{v}t\right)u+u^{5}=0,
	\end{equation}
	and 
	\begin{equation}
	\left\Vert U[t_{0}]-Q[t_{0}]\right\Vert _{\dot{H}^{1}\times L^{2}}\leq\epsilon.
	\end{equation}
	Then there exists free data
	\begin{equation}
	U_{0}[0]=\left(f_{0},g_{0}\right)^{t}\in\dot{H}^{1}\times L^{2}
	\end{equation}
	such that
	\begin{equation}
	\left\Vert U[t]-Q[t]-e^{tJH_{F}}U_{0}[0]\right\Vert _{\dot{H}^{1}\times L^{2}}\rightarrow0,\,\,t\rightarrow\infty.
	\end{equation}
	In other words, the error $u(t)-Q_{1}\left(x\right)-Q_{2}\left(x-\vec{v}t\right)$
	scatters to the free wave.
\end{rem}
We also have the existence of the purely-soliton solution with unstable excited states.
\begin{thm}
	In $\mathbb{R}^{3}$, there exists a solution $u$ to 
	\begin{equation}
	\partial_{tt}u-\Delta u+V_{1}(x)u+V_{2}\left(x-\vec{v}t\right)u+u^{5}=0\label{eq:eqtwo-1}
	\end{equation}
	such that 
	\[
	\lim_{t\rightarrow\infty}\left\Vert U[t]-Q[t]\right\Vert _{\dot{H}^{1}\times L^{2}} = 0.
	\]
\end{thm}
In order to deal with bound states, here we need more complicated
arguments. We will follow an idea based on the weak convergence from
Merle \cite{Merle} and Martel \cite{Mart} which are also used in
many other constructions of multisoltions, for example in \cite{MM,JJ1,JJ2,CM,CM1}. 
\begin{proof}
	We still take $t_{0}$ large enough as before. Taking a sequence $t_{n}\rightarrow\infty$.
	Consider the equation for $h$: 
	\begin{align}
		\partial_{tt}h-\Delta h+h^{5}+\left(V_{1}(x)+5Q_{1}^{4}(x)\right)h+\left(V_{2}\left(x-\vec{v}t\right)+5Q_{2}^{4}\left(x-\vec{v}t\right)\right)h\nonumber \\
		=F_{1}(x,t)+F_{2}(x,t)+F(x,t)+N(h,x,t)+a(x,t)h\label{eq:equationh}
	\end{align}
	We can construct a scattering state $h_{n}$ to equation \eqref{eq:equationh}
	as in the proof of Theorem \ref{thm:condstab} such that 
	\begin{equation}
	\left\Vert h_{n}\left(t_{n}\right)\right\Vert _{\dot{H}^{1}\times L^{2}}\lesssim\frac{1}{\sqrt{t_{n}}}.
	\end{equation}
	Moreover, by the estimates \eqref{eq:RStrichartz} and \eqref{eq:RStrichartzS},
	we have 
	\begin{align}
		\sup_{x\in\mathbb{R}^{3}}\left(\int_{t}^{t_{n}}\left|h_{n}(x,t)\right|^{2}dt\right)^{\frac{1}{2}} & \lesssim\frac{1}{\sqrt{t_{n}}}+\frac{1}{\sqrt{t}},
	\end{align}
	and 
	\begin{align}
		\sup_{x\in\mathbb{R}^{3}}\left(\int_{t}^{t_{n}}\left|h_{n}(x+\vec{v}t,t)\right|^{2}dt\right)^{\frac{1}{2}} & \lesssim\frac{1}{\sqrt{t_{n}}}+\frac{1}{\sqrt{t}}.
	\end{align}
	Furthermore, by a similar argument to the proof of Theorem \ref{thm:existence},
	Lemma \ref{lem:energyL} and Remark \ref{rem:eg}, we can conclude
	that 
	\begin{equation}
	\left\Vert h_{n}(t)\right\Vert _{\dot{H}^{1}\times L^{2}}\lesssim\frac{1}{\sqrt{t_{n}}}+\frac{1}{\sqrt{t}}.
	\end{equation}
	Notice that over $\left[t_{0},t_{n}\right]$, 
	\begin{equation}
	\left\Vert h_{n}(t)\right\Vert _{\dot{H}^{1}\times L^{2}}\lesssim\frac{1}{\sqrt{t_{n}}}+\frac{1}{\sqrt{t}}\lesssim\frac{1}{\sqrt{t}}\label{eq:decayrate}
	\end{equation}
	with a constant independent of $n$.
	
	Then up to passing to a subsequence 
	\begin{equation}
	h_{n}(t_{0})\rightharpoonup h_{0}\in\dot{H}^{1}\times L^{2}
	\end{equation}
	weakly. Let $h$ be a solution of the equation \eqref{eq:equationh}
	with $h_{0}$ as the initial data at $t=t_{0}$. By the weak continuity
	of the flow, for example in \cite{BH,JJ1,JLX}, one can obtain that
	$h$ exists on the time interval from $[t_{0},\infty)$ and for $t\in[t_{0},\infty)$,
	\begin{equation}
	h_{n}(t)\rightharpoonup h(t)\in\dot{H}^{1}\times L^{2}.
	\end{equation}
	Then passing to the weak limit in \eqref{eq:decayrate}, one has 
	\begin{equation}
	\left\Vert h(t)\right\Vert _{\dot{H}^{1}\times L^{2}}\lesssim\frac{1}{\sqrt{t}}.
	\end{equation}
Therefore, we conclude that if we write 
	\begin{equation}
	U[t]=\left(u,u_{t}\right)^{t},\,Q[t]=\left(Q_{1}\left(x\right)-Q_{2}\left(x-\vec{v}t\right),\,\partial_{t}\left(Q_{1}\left(x\right)-Q_{2}\left(x-\vec{v}t\right)\right)\right)^{t},
	\end{equation}
	there exists a solution $u$ to 
	\begin{equation}
	\partial_{tt}u-\Delta u+V_{1}\left(x\right)u+V_{2}\left(x-\vec{v}t\right)u+u^{5}=0,
	\end{equation}
	such that 
	\begin{equation}
	\lim_{t\rightarrow\infty}\left\Vert U[t]-Q[t]\right\Vert _{\dot{H}^{1}\times L^{2}}\rightarrow0.
	\end{equation}
	Moreover, we have the decay rate 
	\begin{equation}
	\left\Vert U[t]-Q[t]\right\Vert _{\dot{H}^{1}\times L^{2}}\lesssim\frac{1}{\sqrt{t}}.
	\end{equation}
	We are done.
\end{proof}
\begin{rem}
	As in Remark \ref{rem:general1}, the above construction holds for the general
	case.
\end{rem}

\section{Appendix: Linear Theory}

In this Appendix, we recall the results from Chen \cite{GC3} on wave
equations with a charge transfer Hamiltonian in $\mathbb{R}^{3}$.
In order to handle the strong interaction of solitons in our nonlinear
application, we also need some refined version of inhomogeneous reversed
Strichartz estimates. 

\subsection{Charge transfer model}

Before we give the precise definition of our model, it is necessary
to introduce Lorentz transformations. Given a vector $\vec{\mu}\in\mathbb{R}^{3}$,
there is a Lorentz transformation $L\left(\vec{\mu}\right)$ acting
on $\left(x,t\right)\in\mathbb{R}^{3+1}$ such that it makes the moving
frame $\left(x-\vec{\mu}t,t\right)$ stationary. We can use a $4\times4$
matrix $B(\vec{\mu})$ to represent the transformation $L\left(\vec{\mu}\right)$.
Moreover, for the given vector $\vec{\mu}=(\mu_{1},\mu_{2},\mu_{3})\in\mathbb{R}^{3}$,
there is a $3\times4$ matrix $M\left(\vec{\mu}\right)$ such that
\begin{equation}
\left(x-\vec{\mu}t\right)^{T}=M\left(\vec{\mu}\right)\left(x,t\right)^{T},
\end{equation}
where the superscript $T$ denotes the transpose of a vector.

With the preparations above, we can set up our model. We consider
the scalar charge transfer model for wave equations in the following
sense:
\begin{defn}
	\label{def: Charge} By a wave equation with a charge transfer Hamiltonian
	we mean a wave equation 
	\begin{equation}
	\partial_{tt}u-\Delta u+\sum_{j=1}^{m}\mathbf{\mathrm{V}}_{j}\left(x-\vec{v}_{j}t\right)u=0,\label{eq:18}
	\end{equation}
	\[
	u|_{t=0}=f,\,\,\partial_{t}u|_{t=0}=g,\,\,x\in\mathbb{R}^{3},
	\]
	where $\vec{v}_{j}$'s are distinct vectors in $\mathbb{R}^{3}$ with
	\begin{equation}
	\left|\vec{v}_{i}\right|<1,\,1\leq i\leq m.
	\end{equation}
	and the real potentials $\mathbf{\mathrm{V}}_{j}$ are such that $\forall1\leq j\leq m$ 
	
	1) $\mathbf{\mathrm{V}}_{j}$ is time-independent and decays with rate $\left\langle x\right\rangle ^{-\alpha}$
	with $\alpha>3$
	
	2) $0$ is neither an eigenvalue nor a resonance of the operators
	\begin{equation}
	H_{j}=-\Delta+\mathbf{\mathrm{V}}_{j}\left(S\left(\vec{v}_{j}\right)x\right),\label{eq:19}
	\end{equation}
	where $S\left(\vec{v_{j}}\right)x=M\left(\vec{v}_{j}\right)B^{-1}\left(\vec{v}_{j}\right)\left(x,0\right)^{T}.$ 
\end{defn}
Recall that $\psi$ is a resonance at $0$ if it is a distributional
solution of the equation $H_{k}\psi=0$ which belongs to the space
$L^{2}\left(\left\langle x\right\rangle ^{-\sigma}dx\right):=\left\{ f:\,\left\langle x\right\rangle ^{-\sigma}f\in L^{2}\right\} $
for any $\sigma>\frac{1}{2}$, but not for $\sigma=\frac{1}{2}.$
\begin{rem*}
	The construction of $S\left(\vec{v}_{j}\right)$ is clear from the
	change between different frames under Lorentz transformations. In
	our concrete problem below \eqref{eq:2V}, $S\left(\vec{v}_{j}\right)$
	can be written down explicitly. 
\end{rem*}
To be consistent with our nonlinear application, throughout this section,
we discuss the wave equation with a charge transfer Hamiltonian in
the sense of Definition \ref{def: Charge} with $m=2$, a stationary
$\mathbf{\mathrm{V}}_{1}$ and a $\mathbf{\mathrm{V}}_{2}$ moving along $\overrightarrow{e_{1}}$ with
speed $\left|v\right|<1$, i.e., the velocity is 
\begin{equation}
\vec{v}=\left(v,0,0\right).
\end{equation}
Under this setting, by Definition \ref{def: Charge}, 
\begin{equation}
H_{1}=-\Delta+\mathbf{\mathrm{V}}_{1}(x),
\end{equation}
and 
\begin{equation}
\,H_{2}=-\Delta+\mathbf{\mathrm{V}}_{2}\left(\sqrt{1-\left|v\right|^{2}}x_{1},x_{2},x_{3}\right).\label{eq:2V}
\end{equation}

An indispensable tool we need to study the charge transfer model is
the Lorentz transformation. Again, we apply Lorentz transformations
$L$ with respect to a moving frame with speed $\left|v\right|<1$
along the $x_{1}$ direction. After we apply the Lorentz transformation
$L$, under the new coordinates, $\mathbf{\mathrm{V}}_{2}$ is stationary meanwhile
$\mathbf{\mathrm{V}}_{1}$ will be moving.

Writing down the Lorentz transformation $L$ explicitly, we have 
\begin{equation}
\begin{cases}
t'=\gamma\left(t-vx_{1}\right)\\
x_{1}'=\gamma\left(x_{1}-vt\right)\\
x_{2}'=x_{2}\\
x_{3}'=x_{3}
\end{cases}\label{eq:LorentzT}
\end{equation}
with 
\begin{equation}
\gamma=\frac{1}{\sqrt{1-\left|v\right|^{2}}}.
\end{equation}
We can also write down the inverse transformation of the above:
\begin{equation}
\begin{cases}
t=\gamma\left(t'+vx_{1}'\right)\\
x_{1}=\gamma\left(x_{1}'+vt'\right)\\
x_{2}=x_{2}'\\
x_{3}=x_{3}'
\end{cases}.\label{eq:InvLorentT}
\end{equation}
Under the Lorentz transformation $L$, if we use the subscript $L$
to denote a function with respect to the new coordinate $\left(x',t'\right)$,
we have 
\begin{equation}
u_{L}\left(x_{1}',x_{2}',x_{3}',t'\right)=u\left(\gamma\left(x_{1}'+vt'\right),x_{2}',x_{3}',\gamma\left(t'+vx_{1}'\right)\right)\label{eq:Lcoordinate}
\end{equation}
and 
\begin{equation}
u(x,t)=u_{L}\left(\gamma\left(x_{1}-vt\right),x_{2},x_{3},\gamma\left(t-vx_{1}\right)\right).\label{eq:ILcoordinate}
\end{equation}

\subsection{Strichartz estimates}

With the above preparations, we recall some important results from
Chen \cite{GC3}. Adapting the linear model to our nonlinear setting,
we consider the following problem.

Suppose $u$ solves
\begin{equation}
\partial_{tt}u-\Delta u+\mathbf{\mathrm{V}}_{1}(x)u+\mathbf{\mathrm{V}}_{2}(x-\vec{v} t)u=F+F_{1}+F_{2}\label{eq:chargeeq}
\end{equation}
with initial data 
\begin{equation}
u(x,0)=f(x),\,u_{t}(x,0)=g(x).
\end{equation}

Let $w_{1},\,\ldots,\,w_{m}$ and $m_{1},\,\ldots,\,m_{\ell}$ be
the normalized bound states of $H_{1}$ and $H_{2}$ associated with
the negative eigenvalues $-\lambda_{1}^{2},\,\ldots,\,-\lambda_{m}^{2}$
and $-\mu_{1}^{2},\,\ldots,\,-\mu_{\ell}^{2}$ respectively (notice
that by our assumptions, $0$ is not an eigenvalue). In other words,
we assume 
\begin{equation}
H_{1}w_{i}=-\lambda_{i}^{2}w_{i},\,\,\,w_{i}\in L^{2},\,\lambda_{i}>0.
\end{equation}

\begin{equation}
H_{2}m_{i}=-\mu_{i}^{2}m_{i},\,\,\,m_{i}\in L^{2},\,\mu_{i}>0.
\end{equation}
We denote by $P_{b}\left(H_{1}\right)$ and $P_{b}\left(H_{2}\right)$
the projections on the the bound states of $H_{1}$ and $H_{2}$,
respectively, and let $P_{c}\left(H_{i}\right)=Id-P_{b}\left(H_{i}\right),\,i=1,2$.
To be more explicit, we have 
\begin{equation}
P_{b}\left(H_{1}\right)=\sum_{i=1}^{m}\left\langle \cdot,w_{i}\right\rangle w_{i},\,\,\,\,\,P_{b}\left(H_{2}\right)=\sum_{j=1}^{\ell}\left\langle \cdot,m_{j}\right\rangle m_{j}.
\end{equation}
In order to study the equation with time-dependent potentials, we
need to introduce a suitable projection. Again, with Lorentz transformations
$L$ associated with the moving potential $\mathbf{\mathrm{V}}_{2}(x-\vec{v}t)$, we use the
subscript $L$ to denote a function under the new frame $\left(x',t'\right)$. 
\begin{defn}[Scattering states]
	\label{AO}Let 
	\begin{equation}
	\partial_{tt}u-\Delta u+\mathbf{\mathrm{V}}_{1}(x)u+\mathbf{\mathrm{V}}_{2}(x-\vec{v}t)u=F+F_{1}+F_{2},\label{eq:eqBSsec-1}
	\end{equation}
	with initial data
	\begin{equation}
	u(x,0)=f(x),\,u_{t}(x,0)=g(x).
	\end{equation}
	If $u$ also satisfies 
	\begin{equation}
	\left\Vert P_{b}\left(H_{1}\right)u(t)\right\Vert _{L_{x}^{2}}\rightarrow0,\,\,\left\Vert P_{b}\left(H_{2}\right)u_{L}(t')\right\Vert _{L_{x'}^{2}}\rightarrow0\,\,\,t,t'\rightarrow\infty,\label{eq:ao2-1}
	\end{equation}
	we call it a scattering state. 
\end{defn}
Define the space $I$ as
\begin{equation}
I=\left\{ G(x,t)\in L_{x}^{\frac{3}{2},1}L_{t}^{2}\bigcap L_{x_{1}}^{1}L_{\widehat{x_{1}}}^{2,1}L_{t}^{2}\,\text{such that}\,\,\left\Vert \left\langle x\right\rangle ^{\frac{1}{2}+\epsilon}G\right\Vert _{L_{t,x}^{2}}<\infty\right\} \label{eq:Ispace}
\end{equation}
for the strong interactions terms where $\widehat{x_{1}}$ is the
subspace orthogonal to the $x_{1}$ direction. Define 
\begin{equation}
\left\Vert G\right\Vert _{I}=\max\left\{ \,\left\Vert \left\langle x\right\rangle ^{\frac{1}{2}+\epsilon}G\right\Vert _{L_{t,x}^{2}},\,\left\Vert G\right\Vert _{L_{x_{1}}^{1}L_{\widehat{x_{1}}}^{2,1}L_{t}^{2}},\,\left\Vert G\right\Vert _{L_{x}^{\frac{3}{2},1}L_{t}^{2}}\right\} .\label{eq:Inorm}
\end{equation}
Also recall that for a function $G(x,t)$, we use the notation: 
\begin{equation}
G^{S}(x,t):=G(x+\vec{v}t,t).
\end{equation}
First of all, we have Strichartz estimates:
\begin{thm}
	\label{thm:StriCharB}Suppose $u$ is a scattering
	state in the sense of Definition \ref{AO}. Then for $p>2$ and $(p,q)$
	satisfying 
	\begin{equation}
	\frac{1}{2}=\frac{1}{p}+\frac{3}{q},
	\end{equation}
	we have
	\begin{align}
		\|u\|_{L_{t}^{p}\left([0,\infty),\,L_{x}^{q}\right)} & \lesssim\|g\|_{L^{2}}+\|f\|_{\dot{H}^{1}}+\left\Vert F\right\Vert _{L_{t}^{1}L_{x}^{2}}\nonumber \\
		& +\left\Vert F_{1}\right\Vert _{I}+\left\Vert F_{2}^{S}\right\Vert _{I}.\label{eq:Strichartz}
	\end{align}
	and 
	\begin{align}
		\|u\|_{L_{t}^{2}\left([0,\infty),\,L_{r}^{\infty}L_{\omega}^{2}\right)} & \lesssim\|g\|_{L^{2}}+\|f\|_{\dot{H}^{1}}+\left\Vert F\right\Vert _{L_{t}^{1}L_{x}^{2}}\nonumber \\
		& +\left\Vert F_{1}\right\Vert _{I}+\left\Vert F_{2}^{S}\right\Vert _{I}.\label{eq:EndStrichartz}
	\end{align}
\end{thm}
Secondly, one has the energy estimate:
\begin{thm}
	\label{thm:EnergyCharge}Suppose $u$ is a scattering
	state in the sense of Definition \ref{AO}, then we have
	\begin{align}
		\sup_{t\text{\ensuremath{\in\mathbb{R}}}}\left(\|\nabla u(t)\|_{L^{2}}+\|u_{t}(t)\|_{L^{2}}\right) & \lesssim\|g\|_{L^{2}}+\|f\|_{\dot{H}^{1}}+\left\Vert F\right\Vert _{L_{t}^{1}L_{x}^{2}}\nonumber \\
		& +\left\Vert F_{1}\right\Vert _{I}+\left\Vert F_{2}^{S}\right\Vert _{I}.\label{eq:EnergyI}
	\end{align}
\end{thm}
Even more importantly, we obtain the endpoint reversed Strichartz
estimates for $u$.
\begin{thm}
	\label{thm:EndRStriCB}Suppose
	$u$ is a scattering state in the sense of Definition \ref{AO}, then
	\begin{align}
		\sup_{x\in\mathbb{R}^{3}}\left(\int_{0}^{\infty}\left|u(x,t)\right|^{2}dt\right)^{\frac{1}{2}} & \lesssim\|g\|_{L^{2}}+\|f\|_{\dot{H}^{1}}+\left\Vert F\right\Vert _{L_{t}^{1}L_{x}^{2}}\nonumber \\
		& +\left\Vert F_{1}\right\Vert _{I}+\left\Vert F_{2}^{S}\right\Vert _{I},\label{eq:RStrichartz}
	\end{align}
	and 
	\begin{align}
		\sup_{x\in\mathbb{R}^{3}}\left(\int_{0}^{\infty}\left|u(x+\vec{v}t,t)\right|^{2}dt\right)^{\frac{1}{2}} & \lesssim\|g\|_{L^{2}}+\|f\|_{\dot{H}^{1}}+\left\Vert F\right\Vert _{L_{t}^{1}L_{x}^{2}}\nonumber \\
		& +\left\Vert F_{1}\right\Vert _{I}+\left\Vert F_{2}^{S}\right\Vert _{I}.\label{eq:RStrichartzS}
	\end{align}
	Moreover, one has
	
	\begin{align}
		\left\Vert \left\langle x\right\rangle ^{-3}u(x,t)\right\Vert _{L_{x}^{\frac{3}{2},1}L_{t}^{\infty}\bigcap L_{x_{1}}^{1}L_{\widehat{x_{1}}}^{2,1}L_{t}^{\infty}} & \lesssim\|g\|_{L^{2}}+\|f\|_{\dot{H}^{1}}+\left\Vert F\right\Vert _{L_{t}^{1}L_{x}^{2}}\nonumber \\
		& +\left\Vert F_{1}\right\Vert _{D_{1}}+\left\Vert F_{2}^{S}\right\Vert _{D_{1}}\label{eq:Localinfty}
	\end{align}
	\begin{align}
		\left\Vert \left\langle x\right\rangle ^{-3}u(x+\vec{v}t,t)\right\Vert _{L_{x}^{\frac{3}{2},1}L_{t}^{\infty}\bigcap L_{x_{1}}^{1}L_{\widehat{x_{1}}}^{2,1}L_{t}^{\infty}} & \lesssim\|g\|_{L^{2}}+\|f\|_{\dot{H}^{1}}+\left\Vert F\right\Vert _{L_{t}^{1}L_{x}^{2}}\nonumber \\
		& +\left\Vert F_{1}\right\Vert _{D_{1}}+\left\Vert F_{2}^{S}\right\Vert _{D_{1},}\label{eq:localinftyS}
	\end{align}
	where 
	\begin{equation}
	D_{1}:=\left\{ G(x,t)\in L_{x}^{\frac{3}{2},1}L_{t}^{1}\bigcap L_{x_{1}}^{1}L_{\widehat{x_{1}}}^{2,1}L_{t}^{1}\bigcap L_{t}^{2}L_{x}^{2}\right\} 
	\end{equation}
	and 
	\begin{equation}
	\left\Vert G\right\Vert _{D_{1}}:=\max\left\{ \left\Vert G\right\Vert _{L_{x}^{\frac{3}{2},1}L_{t}^{1}},\,\left\Vert G\right\Vert _{L_{x_{1}}^{1}L_{\widehat{x_{1}}}^{2,1}L_{t}^{1}},\,\left\Vert G\right\Vert _{L_{t}^{2}L_{x}^{2}}\right\} .
	\end{equation}
	One can replace $D_{1}$ by 
	\begin{equation}
	D_{2}:=\left\{ G(x,t)\in L_{x}^{\frac{3}{2},1}L_{t}^{\infty}\bigcap L_{x_{1}}^{1}L_{\widehat{x_{1}}}^{2,1}L_{t}^{\infty}\bigcap L_{t}^{2}L_{x}^{2}\right\} 
	\end{equation}
	and 
	\begin{equation}
	\left\Vert G\right\Vert _{D_{2}}:=\max\left\{ \left\Vert G\right\Vert _{L_{x}^{\frac{3}{2},1}L_{t}^{\infty}},\,\left\Vert G\right\Vert _{L_{x_{1}}^{1}L_{\widehat{x_{1}}}^{2,1}L_{t}^{\infty}},\,\left\Vert G\right\Vert _{L_{t}^{2}L_{x}^{2}}\right\} .
	\end{equation}
\end{thm}

\subsection{Energy comparison}

Next, we recall the energy comparison for wave equations with respect
to different Lorentz frames. 

Following Chen \cite{GC2,GC3}, we consider wave equations with time-dependent
potentials
\begin{equation}
\partial_{tt}u-\Delta u+V(x,t)u=F
\end{equation}
with
\begin{equation}
\left|V(x,\mu x_{1})\right|\lesssim\frac{1}{\left\langle x\right\rangle ^{2}}
\end{equation}
uniformly for $0\leq\left|\mu\right|\leq1$. These in particular
apply to wave equations with moving potentials with speed strictly
less than $1$. For example, if the potential is of the form 
\begin{equation}
V(x,t)=V\left(x-\vec{v}t\right)
\end{equation}
with 
\begin{equation}
\left|V(x)\right|\lesssim\frac{1}{\left\langle x\right\rangle ^{2}}
\end{equation}
then it is transparent that 
\begin{equation}
\left|V(x,\mu x_{1})\right|=\left|V(x-\vec{v}\mu x_{1})\right|\lesssim\frac{1}{\left\langle x\right\rangle ^{2}}.
\end{equation}
We sketch the argument in \cite{GC2}, suppose
\begin{equation}
\partial_{tt}u-\Delta u+V(x,t)u=F,
\end{equation}
then it is clear that
\begin{eqnarray}
Fu_{t} & = & u_{t}\left(\square u-V(t)u\right)\nonumber \\
& = & -\partial_{t}\left(\frac{\left|u_{t}\right|^{2}}{2}+\frac{\left|u_{x}\right|^{2}}{2}\right)+\mathrm{div}\left(\nabla uu_{t}\right)-V(x,t)uu_{t}.\label{eq:vector}
\end{eqnarray}
We apply the space-time divergence theorem to 
\begin{equation}
\left(\nabla uu_{t},-\left(\frac{\left|u_{t}\right|^{2}}{2}+\frac{\left|u_{x}\right|^{2}}{2}\right)\right)
\end{equation}
then one has the following comparison with the inhomogeneous term,
see Chen \cite{GC2}.
\begin{thm}
	\label{thm:generalC}Let $\left|v\right|<1$. Suppose 
	\begin{equation}
	\partial_{tt}u-\Delta u+V(x,t)u=F(x,t)
	\end{equation}
	and 
	\begin{equation}
	\left|V(x,\mu x_{1})\right|\lesssim\frac{1}{\left\langle x\right\rangle ^{2}}
	\end{equation}
	for $0\leq\left|\mu\right|<1$. Then 
	
	\begin{eqnarray}
	\int\left|\nabla_{x}u\left(x_{1},x_{2},x_{3},vx_{1}\right)\right|^{2}+\left|\partial_{t}u\left(x_{1},x_{2},x_{3},vx_{1}\right)\right|^{2}dx\nonumber \\
	\text{\ensuremath{\lesssim}}\int\left|\nabla_{x}u\left(x_{1},x_{2},x_{3},0\right)\right|^{2}+\left|\partial_{t}u\left(x_{1},x_{2},x_{3},0\right)\right|^{2}dx\label{eq:generalC}\\
	+\int_{\mathbb{R}}\int_{\mathbb{R}^{3}}\left|F(x,t)\right|^{2}dxdt\nonumber 
	\end{eqnarray}
	and 
	\begin{eqnarray}
	\int\left|\nabla_{x}u\left(x_{1},x_{2},x_{3},0\right)\right|^{2}+\left|\partial_{t}u\left(x_{1},x_{2},x_{3},0\right)\right|^{2}dx\nonumber \\
	\text{\ensuremath{\lesssim}}\int\left|\nabla_{x}u\left(x_{1},x_{2},x_{3},vx_{1}\right)\right|^{2}+\left|\partial_{t}u\left(x_{1},x_{2},x_{3},vx_{1}\right)\right|^{2}dx\label{eq:generalC-1}\\
	+\int_{\mathbb{R}}\int_{\mathbb{R}^{3}}\left|F(x,t)\right|^{2}dxdt\nonumber 
	\end{eqnarray}
	where the implicit constant depends on $v$ and $V$.
\end{thm}
From the theorem above, we know that the initial energy with respect to different
frames stays comparable up to $\left\Vert F\right\Vert _{L_{t,x}^{2}}$.

\subsection{Scattering}

In this subsection, we discuss the scattering behavior of the solution
to the nonlinear equation for $h$: 
\begin{align}
	\partial_{tt}h-\Delta h+h^{5}+\left(V_{1}(x)+5W_{1}^{4}(x)\right)h+\left(V_{2}\left(x-\vec{v}t\right)+5W_{2}^{4}\left(x-\vec{v}t\right)\right)h\nonumber \\
	=F_{1}(x,t)+F_{2}(x,t)+F(x,t)+N(h,x,t)+a(x,t)h & .\label{eq:equationh-1}
\end{align}
We show that if $h$ is bounded in the $S$ norm where

\begin{equation}
S=\left\{ \left\Vert u\right\Vert _{Stri},\,\left\Vert u\right\Vert _{\dot{H}^{1}\times L^{2}},\,\left\Vert u\right\Vert _{L_{x}^{\infty}L_{t}^{2}},\,\left\Vert u^{S}\right\Vert _{L_{x}^{\infty}L_{t}^{2}},\,\left\Vert u\right\Vert _{D},\,\left\Vert u^{S}\right\Vert _{D}<\infty\right\} \label{eq:Sspace-1}
\end{equation}
and
\begin{equation}
D:=\left\{ \left\langle x\right\rangle ^{-3}G(x,t)\in L_{x}^{\frac{3}{2},1}L_{t}^{\infty}\bigcap L_{x_{1}}^{1}L_{\widehat{x_{1}}}^{2,1}L_{t}^{\infty}\right\} ,\label{eq:Dspace-1}
\end{equation}
then $h$ scatters to a free wave. 

We will use the notations from the introduction.
\begin{thm}
	\label{thm:scattering}Suppose that $h$ is a solution to \eqref{eq:equationh-1}
	such that 
	\begin{equation}
	\left\Vert h\right\Vert _{S}<\infty,\,\left\Vert \left\langle x\right\rangle ^{\frac{1}{2}+\epsilon}F_{1}\right\Vert _{L_{t,x}^{2}}<\infty,\left\Vert \left\langle x\right\rangle ^{\frac{1}{2}+\epsilon}F_{2}^{S}\right\Vert _{L_{t,x}^{2}}<\infty\,\text{and }\,\left\Vert F\right\Vert _{L_{t}^{1}L_{x}^{2}}<\infty.
	\end{equation}
	Write 
	\begin{equation}
	H[t]=\left(h,h_{t}\right)^{t}\in C^{0}\left([0,\infty);\,\dot{H}^{1}\right)\times C^{0}\left([0,\infty);\,L^{2}\right),
	\end{equation}
	with initial data \textup{$H[0]=\left(f,g\right)^{t}\in\dot{H}^{1}\times L^{2}$.
		Then there exists free data 
		\[
		H_{0}[0]=\left(f_{0},g_{0}\right)^{t}\in\dot{H}^{1}\times L^{2}
		\]
		such that 
		\begin{equation}
		\left\Vert H[t]-e^{tJH_{F}}H_{0}[0]\right\Vert _{\dot{H}^{1}\times L^{2}}\rightarrow0
		\end{equation}
		as $t\rightarrow\infty$.}
\end{thm}
\begin{proof}
	We set $A=\sqrt{-\Delta}$ and notice that 
	\begin{equation}
	\left\Vert Af\right\Vert _{L^{2}}\simeq\left\Vert f\right\Vert _{\dot{H}^{1}},\,\,\forall f\in C^{\infty}\left(\mathbb{R}^{3}\right).
	\end{equation}
	For real-valued $u=\left(u_{1},u_{2}\right)\in\mathcal{H}=\dot{H}^{1}\left(\mathbb{R}^{3}\right)\times L^{2}\left(\mathbb{\mathbb{R}}^{3}\right)$,
	we write 
	\begin{equation}
	U:=Au_{1}+iu_{2}.
	\end{equation}
	We know 
	\begin{equation}
	\left\Vert U\right\Vert _{L^{2}}\simeq\left\Vert \left(u_{1},u_{2}\right)\right\Vert _{\mathcal{H}}.
	\end{equation}
	We also notice that $h$ solves \eqref{eq:equationh-1} if and only
	if 
	\begin{equation}
	H:=Ah+i\partial_{t}h
	\end{equation}
	satisfies 
	\begin{align}
		i\partial_{t}H & =AH-h^{5}-\left(V_{1}(x)+5W_{1}^{4}(x)\right)h-\left(V_{2}\left(x-\vec{v}t\right)+5W_{2}^{4}\left(x-\vec{v}t\right)\right)h\nonumber \\
		& +F_{1}(x,t)+F_{2}(x,t)+F(x,t)+N(h,x,t)-a(x,t)h\nonumber \\
		& =:AH+D(h,x,t).
	\end{align}
	and
	\begin{equation}
	H[0]=Af+ig\in L^{2}\left(\mathbb{R}^{3}\right).
	\end{equation}
	By Duhamel's formula, for fixed $T$ 
	\begin{equation}
	H[T]=e^{iTA}H[0]-i\int_{0}^{T}e^{-i\left(T-s\right)A}\left(D(h,\cdot,s)\right)\,ds.
	\end{equation}
	Applying the free evolution backwards, we obtain
	\begin{equation}
	e^{-iTA}H[T]=H[0]-i\int_{0}^{T}e^{isA}\left(D(h,\cdot,s)\right)\,ds.
	\end{equation}
	Letting $T$ go to $\infty$, we define 
	\begin{equation}
	H_{0}[0]:=H[0]-i\int_{0}^{\infty}e^{isA}\left(D(h,\cdot,s)\right)\,ds
	\end{equation}
	By construction, we just need to show $H_{0}[0]$ is well-defined
	in $L^{2}$, then automatically, 
	\begin{equation}
	\left\Vert H[t]-e^{tJH_{F}}H_{0}[0]\right\Vert _{L^{2}}\rightarrow0.
	\end{equation}
	It suffices to show 
	\begin{equation}
	\int_{0}^{\infty}e^{isA}\left(D(h,\cdot,s)\right)\,ds\in L^{2}
	\end{equation}
	as $t\rightarrow\infty$.
	
	Recall that 
	\begin{align}
		D(h,\cdot,s) & =-\left(V_{1}(x)+5W_{1}^{4}(x)\right)h-\left(V_{2}\left(x-\vec{v}s\right)+5W_{2}^{4}\left(x-\vec{v}s\right)\right)h\nonumber \\
		& +F_{1}(x,s)+F_{2}(x,s)+F(x,s)+N(h,x,s)-a(x,s)h.
	\end{align}
	We also recall the precise expression of $N$: 
	\begin{align}
		N(h,x,t) & :=\left(10W_{1}^{3}(x)+30W_{1}^{2}(x)W_{2}\left(x-\vec{v}t\right)+30W_{1}(x)W_{2}^{2}\left(x-\vec{v}t\right)+10W_{2}^{3}(x-\vec{v}t)\right)h^{2}\nonumber \\
		& +\left(10W_{1}^{2}(x)+3W_{1}(x)W_{2}\left(x-\vec{v}t\right)+10W_{2}^{2}\left(x-\vec{v}t\right)\right)h^{3}\\
		& +\left(5W_{1}(x)+5W_{2}\left(x-\vec{v}t\right)\right)h^{4}.\nonumber 
	\end{align}
	Furthermore, we have
	\begin{equation}
	M_{1}(h,x,t):=\left(10W_{1}^{3}(x)+30W_{1}^{2}(x)W_{2}\left(x-\vec{v}t\right)+30W_{1}(x)W_{2}^{2}\left(x-\vec{v}t\right)+10W_{2}^{3}(x-\vec{v}t)\right)h^{2}
	\end{equation}
	and 
	\begin{equation}
	M_{1,1}(h,x,t)=10W_{1}^{3}(x)h^{2},\,M_{1,2}(h,x,t)=30W_{1}^{2}(x)W_{2}\left(x-\vec{v}t\right)h^{2}
	\end{equation}
	\begin{equation}
	M_{1,3}(h,x,t)=30W_{1}(x)W_{2}^{2}\left(x-\vec{v}t\right)h^{2},\,M_{1,4}(h,x,t)=10W_{2}^{3}(x-\vec{v}t)h^{2}.
	\end{equation}
	Also we use the notation:
	\begin{equation}
	M_{2}(h,x,t):=\left(10W_{1}^{2}(x)+3W_{1}(x)W_{2}\left(x-\vec{v}t\right)+10W_{2}^{2}\left(x-\vec{v}t\right)\right)h^{3},
	\end{equation}
	\begin{equation}
	M_{3}(h,x,t):=\left(5W_{1}(x)+5W_{2}\left(x-\vec{v}t\right)\right)h^{4}
	\end{equation}
	We estimate each piece separately. By the identical argument as Lemma
	\ref{lem:energyL}, we have 
	\begin{equation}
	\left\Vert \int_{0}^{\infty}e^{isA}\left(\left(V_{1}(x)+5W_{1}^{4}(x)\right)h\right)\,ds\right\Vert _{L^{2}}\lesssim\left\Vert h\right\Vert _{L_{x}^{\infty}L_{t}^{2}}\lesssim\left\Vert h\right\Vert _{S}.
	\end{equation}
	\begin{equation}
	\left\Vert \int_{0}^{\infty}e^{isA}\left(\left(V_{2}\left(x-\vec{v}s\right)+5W_{2}^{4}\left(x-\vec{v}s\right)\right)h\right)\,ds\right\Vert _{L^{2}}\lesssim\left\Vert h^{S}\right\Vert _{L_{x}^{\infty}L_{t}^{2}}\lesssim\left\Vert h\right\Vert _{S}.
	\end{equation}
	
	\begin{equation}
	\left\Vert \int_{0}^{\infty}e^{isA}\left(F_{1}(x,s)\right)\,ds\right\Vert _{L^{2}}\lesssim\left\Vert \left\langle x\right\rangle ^{\frac{1}{2}+\epsilon}F_{1}\right\Vert _{L_{t,x}^{2}}
	\end{equation}
	\begin{equation}
	\left\Vert \int_{0}^{\infty}e^{isA}\left(F_{2}(x,s)\right)\,ds\right\Vert _{L^{2}}\lesssim\left\Vert \left\langle x\right\rangle ^{\frac{1}{2}+\epsilon}F_{2}^{S}\right\Vert _{L_{t,x}^{2}}
	\end{equation}
	\begin{equation}
	\left\Vert \int_{0}^{\infty}e^{isA}\left(M_{1,1}(h,x,t)\right)\,ds\right\Vert _{L^{2}}\lesssim\left\Vert h\right\Vert _{L_{x}^{\infty}L_{t}^{2}}\left\Vert h\right\Vert _{D}\lesssim\left\Vert h\right\Vert _{S}^{2}
	\end{equation}
	\begin{equation}
	\left\Vert \int_{0}^{\infty}e^{isA}\left(M_{1,4}(h,x,t)\right)\,ds\right\Vert _{L^{2}}\lesssim\left\Vert h^{S}\right\Vert _{L_{x}^{\infty}L_{t}^{2}}\left\Vert h^{S}\right\Vert _{D}\lesssim\left\Vert h\right\Vert _{S}^{2}
	\end{equation}
	And by trivial energy estimate for the free evolution:
	\begin{equation}
	\left\Vert \int_{0}^{\infty}e^{isA}\left(F(x,s)\right)\,ds\right\Vert _{L^{2}}\lesssim\left\Vert F\right\Vert _{L_{t}^{1}L_{x}^{2}}
	\end{equation}
	\begin{equation}
	\left\Vert \int_{0}^{\infty}e^{isA}\left(h^{5}\right)\,ds\right\Vert _{L^{2}}\lesssim\left\Vert h\right\Vert _{L_{t}^{5}L_{x}^{10}}^{5}\lesssim\left\Vert h\right\Vert _{S}^{5}
	\end{equation}
	\begin{equation}
	\left\Vert \int_{0}^{\infty}e^{isA}\left(a(x,s)h\right)\,ds\right\Vert _{L^{2}}\lesssim\left\Vert a(x,t)h\right\Vert _{L_{t}^{1}L_{x}^{2}}
	\end{equation}
	\begin{equation}
	\left\Vert \int_{0}^{\infty}e^{isA}\left(M_{2}(h,x,t)\right)\,ds\right\Vert _{L^{2}}\lesssim\left\Vert M_{2}(h,x,t)\right\Vert _{L_{t}^{1}L_{x}^{2}}
	\end{equation}
	\begin{equation}
	\left\Vert \int_{0}^{\infty}e^{isA}\left(M_{3}(h,x,t)\right)\,ds\right\Vert _{L^{2}}\lesssim\left\Vert M_{3}(h,x,t)\right\Vert _{L_{t}^{1}L_{x}^{2}}
	\end{equation}
	\begin{equation}
	\left\Vert \int_{0}^{\infty}e^{isA}\left(M_{1,2}(h,x,t)\right)\,ds\right\Vert _{L^{2}}\lesssim\left\Vert M_{1,2}(h,x,t)\right\Vert _{L_{t}^{1}L_{x}^{2}}
	\end{equation}
	\begin{equation}
	\left\Vert \int_{0}^{\infty}e^{isA}\left(M_{1,3}(h,x,t)\right)\,ds\right\Vert _{L^{2}}\lesssim\left\Vert M_{1,3}(h,x,t)\right\Vert _{L_{t}^{1}L_{x}^{2}}
	\end{equation}
	Applying H\"older's inequality and Strichartz estimates, we can estimate
	\begin{equation}
	\left\Vert a(x,t)h\right\Vert _{L_{t}^{1}L_{x}^{2}}\lesssim\left\Vert a\right\Vert _{L_{t}^{\frac{5}{4}}L_{x}^{\frac{5}{2}}}\left\Vert h\right\Vert _{L_{t}^{5}L_{x}^{10}}\lesssim\left\Vert h\right\Vert _{S},
	\end{equation}
	\begin{equation}
	\left\Vert M_{2}(h,x,t)\right\Vert _{L_{t}^{1}L_{x}^{2}}\lesssim\left\Vert h\right\Vert _{L_{t}^{3}L_{x}^{18}}\lesssim\left\Vert h\right\Vert _{S},
	\end{equation}
	\begin{equation}
	\left\Vert M_{3}(h,x,t)\right\Vert _{L_{t}^{1}L_{x}^{2}}\lesssim\left\Vert h\right\Vert _{L_{t}^{4}L_{x}^{12}}\lesssim\left\Vert h\right\Vert _{S},
	\end{equation}
	and 
	\begin{equation}
	\left\Vert M_{1,2}(h,x,t)\right\Vert _{L_{t}^{1}L_{x}^{2}},\,\left\Vert M_{1,3}(h,x,t)\right\Vert _{L_{t}^{1}L_{x}^{2}}\lesssim\left\Vert h\right\Vert _{L_{t}^{4}L_{x}^{12}}\lesssim\left\Vert h\right\Vert _{S}.
	\end{equation}
	Therefore, 
	\begin{align}
		\left\Vert \int_{0}^{\infty}e^{isA}\left(D(h,\cdot,s)\right)\,ds\right\Vert _{L^{2}} & \lesssim\left\Vert \left\langle x\right\rangle ^{\frac{1}{2}+\epsilon}F_{1}\right\Vert _{L_{t,x}^{2}}+\left\Vert \left\langle x\right\rangle ^{\frac{1}{2}+\epsilon}F_{2}^{S}\right\Vert _{L_{t,x}^{2}}\nonumber \\
		& +\left\Vert F\right\Vert _{L_{t}^{1}L_{x}^{2}}+\left\Vert h\right\Vert _{S}+\left\Vert h\right\Vert _{S}^{5}.
	\end{align}
	And hence 
	\begin{equation}
	H_{0}[0]:=H[0]-i\int_{0}^{\infty}e^{isA}\left(D(h,\cdot,s)\right)\,ds\in L^{2}.
	\end{equation}
	We are done.
\end{proof}

\end{document}